\documentclass[12pt,twoside]{amsart}
\usepackage{amsthm}
\usepackage{verbatim}
\usepackage{amsmath}
\usepackage{bm}
\usepackage{txfonts}
\usepackage{a4wide}
\usepackage[latin1]{inputenc}
\usepackage[T1]{fontenc}
\usepackage{times}
\usepackage{amssymb,latexsym}
\usepackage{enumitem}
\usepackage[stable]{footmisc}
\usepackage{color}
\usepackage[colorlinks,linkcolor=black,citecolor=black, urlcolor=black]{hyperref}

\makeatletter
\newcommand{\sumprime}{\if@display\sideset{}{'}\sum%
            \else\sum'\fi}
\makeatother

\begin{document}

\numberwithin{equation}{section}

\newtheorem{theorem}{Theorem}[section]
\newtheorem{proposition}[theorem]{Proposition}
\newtheorem{corollary}[theorem]{Corollary}
\newtheorem{lemma}[theorem]{Lemma}
\newtheorem{observation}[theorem]{Observation}
\newtheorem{definition}{Definition}
\numberwithin{definition}{section} 
\newtheorem{remark}{Remark}
\def\theremark{\unskip}
\newtheorem{kl}{Key Lemma}
\def\thekl{\unskip}
\newtheorem{question}{Question}
\def\thequestion{\unskip}
\newtheorem{example}{Example}
\def\theexample{\unskip}
\newtheorem{problem}{Problem}
\newtheorem{conjecture}{Conjecture}
\def\theconjecture{\unskip}

\address [Bo-Yong Chen] {School of Mathematical Sciences, Fudan University, Shanghai, 200433, China}
\email{boychen@fudan.edu.cn}

\address [Yuanpu Xiong] {School of Mathematical Sciences and Key Laboratory of Intelligent Computing and Applications (Tongji University), Ministry of Education, Tongji University, Shanghai, 200092, CHINA}
\email{ypxiong@tongji.edu.cn}

\address [Liyou Zhang] {School of Mathematical Sciences, Capital Normal University, Beijing, 100048, China}
\email{zhangly@cnu.edu.cn}

\title{Scalar Curvature,  volumes and the Bergman kernel}
\author{Bo-Yong Chen, Yuanpu Xiong and Liyou Zhang}

\date{}

\begin{abstract}
Motivated by the works of Gromov  and LeBrun  in Riemannian geometry,  we study the analogous phenomena in complex geometry. We first show that both $\int_M |S_C^-(g)|^ndV_g$ and ${\rm vol}_g(M)$ (normalized by $S_C(g)\ge -1$) are bounded below by $\frac{(n\pi)^n}{n!}\mathrm{CanVol}(M)$ for any Hermitian metric $g$ on a compact complex $n-$manifold $M$.  Here $S_C$ denotes the Chern scalar curvature,  $S_C^-=\max\{-S_C,0\}$ and ${\rm CanVol}(M)$ is the canonical volume of $M$,  i.e.,  the volume of the canonical line bundle $K_M$.  Moreover, if ${\rm vol}_g(M)=\frac{(n\pi)^n}{n!}\mathrm{CanVol}(M)$ holds for some K\"ahler metric with $S_C\ge -1$,  then it has to be the K\"ahler-Einstein metric of negative scalar curvature. The completely new phenomenon is that if $M$ is a compact K\"{a}hler manifold such that $K_M$ is nef,  then ${\rm MinVol}_C(M)=\mathcal{I}_C(M)=\mathcal I_C^-(M)=\frac{(n\pi)^n}{n!}\mathrm{CanVol}(M)$,  where ${\rm MinVol}_C(M)$ is the infimum of ${\rm vol}_g(M)$ with $S_C(g)\ge -1$ and $\mathcal I_C^-(M)=\inf_g \int_M |S_C^-(g)|^ndV_g$, $\mathcal I_C(M)=\inf_g \int_M |S_C(g)|^ndV_g$. It remains unknown whether the nef condition is superfluous. The answer is positive when $M$ is obtained by blowing up a finite number of points from a projective manifold with big and nef canonical line bundle.  The arguments are based on the asymptotic behaviour of the Bergman kernel  of $mK_M$ as $m\rightarrow \infty$,  the theory of K\"ahler-Ricci flow and singular K\"{a}hler-Einstein metric, as well as a very delicate gluing technique,  using the Burns-Simanca metric.    
\end{abstract}

\maketitle

\tableofcontents

\section{Introduction}
Let us start with  a unified way of posing reasonable questions in Riemannian geometry.  Given a compact $n-$manifold $M$,  define  $\mathcal{RM}(M)$ to be the set of Riemannian metrics $g$ on $M$.   Consider a nonnegative functional $\mathcal F:\mathcal{RM}(M)\rightarrow \mathbb R$ and look for its infimum $\inf_{\mathcal{RM}(M)} \mathcal F$.    According to  Berger (see \cite{BergerBook},  p. 530),  there are several types of  basic questions:
\begin{enumerate}
\item Is $\inf_{\mathcal{RM}(M)} \mathcal F$ zero or positive?  Classify the manifolds $M$ for which $\inf_{\mathcal{RM}(M)} \mathcal F$ is positive and those for
which it is zero.
\item If positive,  then is it attained by some Riemannian metric $g$ on $M$? When these best metrics exist,  try to classify them.  
\item Find the possible values $\inf_{\mathcal{RM}(M)} \mathcal F$ when $M$ runs through all compact
manifolds of a given dimension. Or at least, we can ask if the set of these values is discrete, or if zero is an isolated point of this set.
\end{enumerate}

Of course,  these questions are almost impossible to answer for a general functional $\mathcal F$,  so one has to restrict to certain special,  but crucial cases,  e.g.,  minimal curvature integrals 
$$
\inf_{g\in \mathcal{RM}(M)} \int_M |K(g)|^{n/2}dV_g,\ \    \inf_{g\in {\mathcal{RM}(M)}} \int_M |Ric(g)|^{n/2}dV_g,\ \  \inf_{g\in {\mathcal{RM}(M)}}  \int_M |S(g)|^{n/2}dV_g,
$$
 where $K(g)$,  $Ric(g)$ and $S(g)$ stand for  sectional curvature,   Ricci curvature and scalar curvature respectively.  We consider the $L^{n/2}$ integral instead of general $L^p$ integrals since it is the only scale-invariant case.    Another important example is Gromov's minimal volume (cf.  \cite{GromovVolume})
 $$
 {\rm MinVol}(M):=\inf_{|K(g)|\le 1} {\rm vol}_g(M)=\inf_{g\in {\mathcal{RM}(M)}} {\rm vol}_g(M) \|K(g)\|_{L^\infty}^{n/2}.
 $$
 The Gauss-Bonnet theorem implies ${\rm MinVol}(M)\ge C_n |\chi(M)|$,  where $\chi(M)$ denotes the Euler characteristic of $M$.  Gromov proved the following highly-nontrivial inequality 
 $$
  {\rm MinVol}(M)\ge \frac{\|M\|}{(n-1)^nn!},
 $$
 where $\|M\|$ is the simplicial volume of $M$,  namely,  $\|M\|:=\inf_c \sum |a_i|$,  $c=\sum a_i\sigma_i$ runs through all representations of the fundamental class $[M]\in H_n(M,\mathbb R)$.  Thus the minimal volume is positive whenever the simplicial volume is.  Positivity of the minimal volume provides an obstruction for the collapsing phenomenon in Gromov's convergence theory. Gromov raised the famous Gap Conjecture:
 $$
 \exists\, \varepsilon_n>0\ \text{such that}\ {\rm MinVol}(M)\le \varepsilon_n\Rightarrow  {\rm MinVol}(M)=0
 $$
 (see \cite{CheegerRong}, \cite{Rong} for partial answers).  A beautiful rigidity theorem due to Besson,  Courtois and Gallot \cite{BCG} states that if $g$ is a Riemannian metric on a hyperbolic $n-$manifold $(M,g_0)$ such that $Ric(g)\ge -(n-1)g$,  then ${\rm vol}_g(M)\ge {\rm vol}_{g_0}(M)$ with equality if and only if $g$ is isometric to $g_0$.   Note $|K(g)|\le 1$ implies $Ric(g)\ge -(n-1)g$. Thus the best metric for ${\rm MinVol}(M)$ is the hyperbolic metric $g_0$ in this case.
 
 Among  $K(g)$,  $Ric(g)$ and $S(g)$,  the last is the weakest.    In his lecture notes on scalar curvature  (cf.  \cite{Gromov4},  p.  38--39),  Gromov conjectured there exists $C_n>0$ such that
\begin{equation}\label{eq:Gromov_conj_int}
\int_M |S^-(g)|^{n/2}dV_g \ge C_n \|M\|,\ \ \ S^-(g):=\max\{-S(g),0\}.
\end{equation}
Note that \eqref{eq:Gromov_conj_int} implies ${\rm vol}_g(M)\ge C_n \|M\|$ when $S(g)\ge -1$.  This conjecture looks extremely hard since basic tools,  such as  Bishop-Gromov's comparison theorem and  Cheng-Yau's gradient estimates,   require a lower bound on the Ricci curvature.  However,   LeBrun \cite{LeBrun99} 
was able to show  
\begin{equation}\label{eq:LeBrun0}
\int_M |S(g)|^{2} dV_g \ge  32 \pi^2 K_M^2
\end{equation}
for minimal complex surfaces of general type, i.e., the canonical line bundle $K_M$ is big and nef. Here, $K_M^2 := \int_M c_1(K_M)^2$. Moreover,  this inequality is sharp.  LeBrun's method is based on  Seiberg-Witten theory,   which is not available when the real dimension is greater than $4$.   Actually,  every simply-connected  compact complex $n-$manifolds $M$ with ample $K_M$ and $n\ge 3$ satisfies  
$$
\inf_{g\in \mathcal{RM}(M)} \int_M |S(g)|^{n} dV_g=0,
$$
in view of Theorem 1 in \cite{LeBrun23}.   

The goal of this paper is to consider analogous questions for compact Hermitian manifolds.  Our hope is to find certain connections  linking  differential-geometric and  algebraic-geometric points of view.  To that end, let us consider  a general compact complex $n-$manifold $M$,  equipped with a smooth Hermitian metric $g$ whose K\"ahler form is given by $\omega$.  Let $Ric(g)$ and $Ric(\omega)$ denote the Riemannian Ricci curvature (for the underlying Riemannian metric) and Chern Ricci curvature form  respectively.   Let $S_C(g)$ denote the  Chern scalar curvature of $g$. The canonical volume of $M$ is given by 
\begin{equation}
\label{eq:CanVolDef}
\mathrm{CanVol}(M):= \limsup_{m\rightarrow+\infty}\frac{P_m(M)}{m^n/n!},
\end{equation}
where  $P_m(M):=\dim\Gamma(M,mK_M)$ is the  $m-$th plurigenus of $M$.  It is known from \cite[\S\,6B]{DemaillyBook} that
$$
\mathrm{CanVol}(M)=K_M^n:=\int_Mc_1(K_M)^n,
$$
when $M$ is a K\"ahler manifold and  $K_M$  is nef.  Recall that $M$ is of general type (i.e.,  $K_M$ is big)  iff $\mathrm{CanVol}(M)>0$.  Our first result is given as follows.

\begin{theorem}\label{th:Main0}
If $(M,g)$ is a compact Hermitian $n-$manifold,  then
\begin{equation}\label{eq:LeBrun_1}
\int_M |S_C^-(g)|^{n} dV_g \geq\frac{(n\pi)^n}{n!}\mathrm{CanVol}(M),
\end{equation}
where $S_C^-(g)=\max\{-S_C(g),0\}$.   
Moreover,  if $S_C(g)\ge -1$,  then
\begin{equation}\label{eq:volume_lower_1}
{\rm vol}_g(M)\ge \frac{(n\pi)^n}{n!}\mathrm{CanVol}(M).
\end{equation}
On the other hand,  if equality in \eqref{eq:volume_lower_1} holds for some Hermitian metric $g$, then it has to be a Hermitian-Einstein metric, that is,
\[
Ric(\omega)=\frac{S_C(g)}{n}\omega,
\]
where $\omega$ is the K\"{a}hler form of $g$. Furthermore,  if $n\geq2$ and $g$ is a {\bf K\"ahler} metric, then it has to be a K\"ahler-Einstein metric of constant negative scalar curvature.
\end{theorem}

\begin{remark}
For any  Einstein metric on a Riemannian manifold of dimension greater than $2$,  the (Riemannian) scalar curvature is constant (see Berger \cite[Theorem 277]{BergerBook}).  The same is not true in general for Hermitian manifolds (although true in the K\"ahler case).  
\end{remark}

\begin{corollary}\label{cor:KahlerNoAttain}
Let $M$ be a compact K\"ahler $n-$manifold  such that $K_M$ is big but not ample. Then for any K\"ahler metric $g$ with $S_C(g)\ge -1$,
\begin{equation}\label{eq:volume_lower_2}
{\rm vol}_g(M) > \frac{(n\pi)^n}{n!}\mathrm{CanVol}(M).
\end{equation}
\end{corollary}

\begin{proof}
Suppose on the contrary that \eqref{eq:volume_lower_2} does not hold.  Then equality in \eqref{eq:volume_lower_1} holds for suitable K\"ahler metric $g$,  so that $g$ is  K\"ahler-Einstein of  negative scalar curvature.  Then $K_M$ is ample in view of Kodaira's theorem,  which is a contradiction. 
\end{proof}

Let $\mathcal{HM}(M)$ denote the set of Hermitian metrics on a compact complex $n-$manifold $M$.  It is reasonable to define 
$$
 \mathcal I_C(M):=\inf_{g\in \mathcal{HM}(M)}  \int_M |S_C(g)|^{n} dV_g,\ \ \mathcal I_C^-(M):=\inf_{g\in \mathcal{HM}(M)} \int_M |S_C^-(g)|^{n} dV_g,
$$
$$
 \mathrm{MinVol}_C(M):=\inf_{S_C(g)\ge -1} \mathrm{vol}_g(M).
$$
We have known that under the hypothesis in Corollary \ref{cor:KahlerNoAttain},  $\mathrm{MinVol}_C(M)$ is not attained by a K\"ahler metric.  It remains unknown whether or not it is achieved by some Hermitian metric.

As a direct consequence of Theorem \ref{th:Main0},  we have

\begin{corollary}\label{cor:lower_gap}
\begin{equation}\label{eq:LeBrun_volume}
\min\left\{\mathcal{I}_C(M),\mathrm{MinVol}_C(M)\right\}\ge \mathcal I_C^-(M) \geq\frac{(n\pi)^n}{n!}\mathrm{CanVol}(M).
\end{equation}
\end{corollary} 

\begin{remark}
A deep result proved independently by Hacon-McKernan \cite{HM} and Takayama \cite{Takayama} states that  
$\mathrm{CanVol}(M)\ge C_n>0$ when $M$ is a  projective manifold  of general type.   Thus 
$\mathrm{MinVol}_C(M),  $ $\mathcal{I}_C(M)$ and $\mathcal{I}_C^-(M)
$
are all bounded below by some $C_n>0$.
\end{remark}

It is natural to formulate a  Gromov-type Gap Conjecture as follows:
\begin{equation}\label{eq:gap_conj}
 \exists\, \varepsilon_n>0\ \text{such that}\ {\rm MinVol}_C(M)\le \varepsilon_n\Rightarrow  {\rm MinVol}_C(M)=0.
\end{equation}
In case $M$ is a projective $n-$manifold,  it suffices to verify
$$
{\rm CanVol}(M) =0 \Rightarrow {\rm MinVol}_C(M)=0.
$$
If there exists some Hermitian metric $g$ with $S_C(g)\ge 0$, for instance,  $M$ is a Fano manifold, then $\mathrm{MinVol}_C(M)=0$, since $S_C(tg)=t^{-1}S_C(g)$ and $dV_{tg}=t^ndV_g$. A typical example is the Hopf manifold $M_\alpha:=(\mathbb C^n\setminus \{0\})/\langle \gamma_\alpha\rangle$,  $\gamma_\alpha=\alpha z$,  $0<|\alpha|<1$.  Note that the Hermitian metric $g=|z|^{-2} \sum dz_j\otimes d\bar{z}_j$ on $\mathbb C^n\setminus \{0\}$ is $\langle \gamma_\alpha\rangle-$invariant,  so it induces a Hermitian metric on $M_\alpha$.  A direct calculation shows the Ricci form $Ric(\omega) = 2n i\partial\bar{\partial}\log |z|\ge 0$, where $\omega$ is the K\"{a}hler form of $g$, so $S_C(g)\ge 0$ and $\mathrm{MinVol}_C(M_\alpha)=0$.

\vspace{2mm}

The following result is entirely new,  in contrast to the Riemannian case.  

\begin{theorem}\label{th:Main1}
If $M$ is a compact K\"{a}hler $n-$manifold such that $K_M$ is nef, then 
\begin{equation}\label{eq:LeBrun_equality}
\mathrm{MinVol}_C(M)=\mathcal{I}_C(M)=\mathcal{I}_C^-(M)=\frac{(n\pi)^n}{n!}\mathrm{CanVol}(M).
\end{equation}
\end{theorem}

\begin{corollary}\label{cor:gap}
The gap conjecture \eqref{eq:gap_conj} is true for compact K\"{a}hler manifolds with $K_M$  nef.
\end{corollary}

\begin{corollary}\label{cor:holo_increasing}
Let $M'$ be a compact complex $n-$manifold which admits a generically finite holomorphic map $F$ to a compact K\"{a}hler $n-$manifold $M$ such that $K_M$ is nef.  Then
\begin{eqnarray}
\mathcal I_C(M')&\ge& {\rm deg}(F)\, \mathcal I_C(M), \label{eq:deg_1}\\
\mathcal I^-_C(M')&\ge& {\rm deg}(F)\, \mathcal I^-_C(M), \label{eq:deg_2}\\
\mathrm{MinVol}_C(M')&\ge& {\rm deg}(F)\, \mathrm{MinVol}_C(M). \label{eq:deg_3}
\end{eqnarray}
Moreover, if $F$ is a covering map, then equalities hold.
\end{corollary}

\begin{conjecture}
\eqref{eq:LeBrun_equality} holds for any compact complex $n-$manifold (at least when $K_M$ is big).
\end{conjecture}

To support this conjecture,  we shall show 

\begin{theorem}\label{th:blow_up}
Let $M$ be a compact K\"{a}hler $n-$manifold with  $n\ge 2$ and $\widehat{M}$ be obtained from $M$ by blowing-up a finite number of points. If $K_M$ is big and nef, then \eqref{eq:LeBrun_equality} holds for $\widehat{M}$.
\end{theorem}

\begin{remark}
The proof of Theorem \ref{th:blow_up} also yields $\mathrm{MinVol}_C(\widehat M)=\mathcal{I}_C^-(\widehat{M})=0$ when $K_M$ is nef but not big (see \S\,\ref{sec:blow_up} for details).
\end{remark}

Our analysis relies heavily on  asymptotic behaviour  of the Bergman kernel of $mK_M$ as $m\rightarrow \infty$,  as well as   the  theory of K\"ahler-Ricci flow and existence of the singular K\"ahler-Einstein metric (when $K_M$ is big).   More precisely,   \eqref{eq:LeBrun_1}  follows directly from   the arithmetic-geometric mean inequality and the well-known asymptotic upper bound of the Bergman kernel of $mK_M$ (which follows easily from  the mean-value inequality of plurisubharmonic functions (cf.  \cite{Berman}; see also \cite{BerndtssonBook})).  If $M$ is a projective manifold of general type,  then there exists a singular K\"{a}hler-Einstein metric $g_{KE}$,  which is smooth on some Zariski open subset $U\subset M$ (cf. \cite[Theorem C.1]{ST12}, see also \cite{BEGZ10}).  With the help of  asymptotic lower bound of the Bergman kernel of $mK_M$ with respect to $g_{KE}$,    \eqref{eq:LeBrun_equality} for the big case can be reduced to find a sequence $\{g_j\}$ of smooth Hermitian metrics on $M$ such that
\begin{enumerate}
 \item $g_j\rightarrow g_{KE}$ in $C^\infty_{\mathrm{loc}}-$topology on $U$;
 \item  both $dV_{g_j}$ and $|S_C(g_j)|$ are uniformly bounded on $M$;
 \item $S_C(g_j)\ge -n$.
\end{enumerate}
In case $K_M$ is big and nef,  the K\"ahler-Ricci flow fits well with this purpose.  It remains open whether the above sequence exists for any compact K\"{a}hler manifold of general type. The proof of Theorem \ref{th:blow_up} is based on a rather delicate gluing technique using the Burns-Simanca metric (compare \cite{LeBrun96} and \cite{Szekelyhidi12}).

An even more interesting problem is to find the relations between the canonical volume and the simplicial volume.   For instance,  it is unknown whether there exists a constant $C_n>0$ such that
\begin{equation}\label{eq:CV-SV}
\mathrm{CanVol}(M)\geq C_n\|M\|
\end{equation}
for any compact complex $n-$manifold $M$.  Note that \eqref{eq:CV-SV} would imply a conjecture of Zhang \cite{WZhang} that $\|M\|=0$ provided ${\rm CanVol}(M)=0$.   Gromov \cite{GromovVolume} showed that if the Ricci curvature of a Riemannian metric $g$ is bounded below by $-1$,  then 
$$
  {\rm vol}_g (M)\ge \frac{\|M\|}{(n-1)^nn!}.  
 $$
 This combined with the argument in \S\,6 would imply  \eqref{eq:CV-SV} for projective manifolds of general type,    provided that there exists a sequence $\{g_j\}$ of K\"ahler metrics on $M$ such that
\begin{enumerate}
 \item $g_j\rightarrow g_{KE}$ in $C^\infty_{\mathrm{loc}}-$topology on $U$;
 \item  $dV_{g_j}$ is uniformly bounded on $M$;
 \item  $Ric(\omega_j)\ge -\omega_j$,  where $\omega_j$ is the K\"ahler form of $g_j$.
\end{enumerate}
Using a result of LeBrun \cite{LeBrun01},  we will show in \S\,8 that for certain complex surfaces of general type $g_{KE}$ cannot be approximated by K\"ahler metrics such that both the volume form and the Ricci form are uniformly bounded.    We will return to this topic in a future paper. 

Finally,  recall that  the Kodaira dimension of a compact complex $n-$manifold is defined to be
$$
\kappa (M):=\limsup_{m\rightarrow+\infty}\frac{\log P_m(M)}{\log{m}}.
$$
It is well known that there exists a constant $C>1$ such that
$$
C^{-1} m^{\kappa(M)} \le P_m(M) \le C m^{\kappa(M)}
$$
(see \cite{Ueno}).  
Two interesting questions immediately arise.  The first is to look for  a  geometric bound similar as scalar curvature integral  or minimal volume for the quantity
$$
\limsup_{m\rightarrow+\infty}\frac{P_m(M)}{{m}^{\kappa(M)}}.
$$
The other is to find the relationship between $\kappa(M)$ and the Ricci rank $\mathcal R_g$ of a Hermitian metric $g$,  i.e.,  the maximum of the number of negative eigenvalues of $Ric(\omega)$.   Liu \cite{Liu} proved $\kappa(M)=\mathcal{R}_g$ when $g$ is a K\"ahler metric with nonpositive bisectional curvature   (see also Wu-Zheng \cite{WZ} for real-analytic case).   In general,  we have the following 

\begin{proposition}\label{prop:RicciRank}
If $(M,g)$ is a compact Hermitian $n-$manifold,  then
\[
\kappa(M)\leq \frac{\mathcal{R}_g+2n}{3} = \mathcal{R}_g+\frac{2}{3}(n-\mathcal{R}_g).
\]
\end{proposition}

Proposition \ref{prop:RicciRank} turns out to be a fairly simple consequence of a very precise  asymptotic upper bound for the Bergman kernel,  which is given as follows.  
Consider  a compact Hermitian $n-$manifold $(M,g)$ and a holomorphic line bundle $L$ over $M$.  Given a (possibly singular) Hermitian metric $h=e^{-\phi}$ of $L$,  define the Bergman space of $mL=L^{\otimes m}$ to be
$$
A^2_{g,h}(M,mL):=\left\{s\in \Gamma(M,mL): \int_M |s|^2_{h^{\otimes m}} dV_g<\infty\right\},
$$
where $\Gamma(M,mL)$ denotes the space of holomorphic sections of $mL$ over $M$.  The corresponding Bergman kernel function is given by 
\[
B_{g,h,mL} (x)=\sum |s_j(x)|^2_{h^{\otimes m}}
\]
for an/any orthonormal basis $\{s_j\}$ of $A^2_{g,h}(M,mL)$.   In case $L=K_M$ and $h=(dV_g)^{-1}$,  we simply write $B_m$ instead of $B_{g,(dV_g)^{-1},mK_M}$.  

\begin{theorem}\label{th:Bergman1}
Let $(M,g)$ be a compact Hermitian $n-$manifold.  Let $\lambda_1,\cdots,\lambda_n$ be eigenvalues of the  Ricci form $Ric(\omega)$ associated to $\omega$.  Fix  $A_m>0$ with 
$$
\lim_{m\rightarrow+\infty}A_m=+\infty \ \  \text{and}\ \  \lim_{m\rightarrow+\infty}A_m/m^{1/3}=0.
$$
  Then there exists a constant $C>0$ such that
\begin{equation}\label{eq:Bergman_upper_1_1}
B_{m}(x)
\leq \left(1+\frac{C A_m^{3/2}}{m^{1/2}}\,\right)\,\frac{m^n}{\pi^nA_m^{n-p-q}}\left(\prod_{\lambda_j\neq0}|\lambda_j|\right)\left(\prod_{\lambda_j>0}e^{-\lambda_jA_m}\right)
\end{equation}
holds uniformly on $M$ as $m\rightarrow \infty$.  Here $p$ (resp.  $q$) is the number of positive (resp.  negative) eigenvalues at $x$.  
\end{theorem}

A large literature exists on the asymptotic behaviour of the Bergman kernel of $mL$ as $m\rightarrow \infty$ in case $L$ is ample,  among them the most famous is the so-called Tian-Yau-Catlin-Zelditch asymptotic expansion (see e.g.,   \cite{Catlin},  \cite{MM},  \cite{Tian},  \cite{Zelditch}),  which has deep applications in K\"ahler geometry.  
Although the proof of Theorem \ref{th:Bergman1} is only a refinement of known methods,  the theorem itself might be of independent interest.  For instance,   it gives a new proof of the following classical vanishing result.   

\begin{theorem}[cf.  \cite{KobayashiBook},  Corollary 3.1.16]\label{th:vanishing}
Let $(M,g)$ be a compact Hermitian manifold.  If the Ricci form admits at least one positive eigenvalue everywhere on $M$,  then $P_m(M)=0$.
\end{theorem}

Recently, Theorem \ref{th:vanishing} has been extended substantially by Yang in \cite{Yang21}.

\section{Preliminaries}

\subsection{Chern scalar curvature}
Let $M$ be a complex $n-$manifold.  Let $g$ be a Hermitian metric given locally by  
$
\sum^n_{j,k=1}g_{j\bar{k}}dz_j\otimes {d\bar{z}_k}.
$
Then locally $\omega=i\sum^n_{j,k=1}g_{j\bar{k}}dz_j\wedge d\bar{z}_k$ is the K\"{a}hler form of $g$. The Chern Ricci curvature form of $\omega$ is defined by 
\[
Ric(\omega)=-i\partial\bar{\partial}\log\det(g_{j\bar{k}}).
\]
The trace of $Ric(\omega)=: i \sum R_{j\bar{k}}dz_j\wedge d\bar{z}_k$ is called 
the Chern scalar curvature $S_C(g)$ of $g$,  namely,   
\[
S_C(g)=\mathrm{tr}_g(Ric(\omega))=\sum g^{j\bar{k}}R_{j\bar{k}},
\]
where $(g^{j\bar{k}})=(g_{j\bar{k}})^{-1}$.  It is well-known that the Riemannian scalar curvature equals twice the Chern scalar curvature if $g$ is a K\"ahler metric.   In general,  the relationship between them is much more complicated (compare \cite{LY17}).

\subsection{Singular Hermitian metrics}
Let $M$ be a complex $n-$manifold and $L$ a holomorphic line bundle over $M$.  A singular Hermitian metric of $L$ may be written as 
$$
h=h_0 e^{-\varphi}
$$
where $h_0$ is a smooth Hermitian metric of $L$ and $\varphi\in L^1_{\mathrm{loc}}(M)$.   $L$ is said to be pseudoeffective if there exists a singular Hermitian metric $h$ of $L$ such that
$$
\Theta_h = \Theta_{h_0} + i\partial\bar{\partial} \varphi
$$
is a closed positive current.

The multiplier ideal sheaf $\mathcal I(h)\subset \mathcal O_M(L)$ of $(M,L)$ is defined by
$$
\Gamma(U,\mathcal I(h)):=\left\{f\in \Gamma(U,\mathcal O_U(L)): |f|^2_{h_0} e^{-\varphi}\in L^1_{\rm loc}(U)\right\}.
$$
Following Tsuji \cite{Tsuji92},  $h$ is called an analytic Zariski decomposition if 
\begin{enumerate}
\item $\Theta_h$ is a closed positive current,
\item for every $m\ge 0$,  the natural inclusion
$$
H^0\left(M,mL\otimes \mathcal I(h^{\otimes m})\right ) \rightarrow H^0(M,mL)
$$
is an isomorphism.  
\end{enumerate}

\subsection{K\"ahler-Ricci flow}

Here we collect some basic facts about the K\"ahler-Ricci flow by following  the lecture notes of Song and Weinkove (cf.  \cite{SW},  p. 19, 20, 48).   Let $(M,g_0)$ be a compact K\"ahler $n-$manifold.  A K\"ahler-Ricci flow means the following equation
\begin{equation}\label{eq:KR1}
\frac{\partial \omega_t}{\partial t}=-Ric(\omega_t),\ \ \ \omega_t|_{t=0}=\omega_0,
\end{equation}
and the normalized K\"ahler-Ricci flow is given by
\begin{equation}\label{eq:KR2}
\frac{\partial \omega_t}{\partial t}=-Ric(\omega_t)-\omega_t,\ \ \ \omega_t|_{t=0}=\omega_0,
\end{equation}
where $\omega_t$ is the K\"{a}hler form of $g_t$. We shall focus on \eqref{eq:KR2}.  Then the following properties hold:
\begin{enumerate}
\item There exists a unique solution $g_t$ of \eqref{eq:KR2} on some maximal time interval $[0,T)$ with some $0<T\le \infty$.  If $K_M$ is nef,  then one can take $T=\infty$.
\item Set $C_0:=-\inf_M S_C(g_0)-n$. Then
\begin{equation}\label{eq:ScalarLower}
S_C(g_t) \ge -n - C_0 e^{-t},\ \ \  t\in [0,T).
\end{equation}
\item Let $C_0$ be as above.  Then
\begin{equation}\label{eq:volumeCompare}
dV_{g_t} \le e^{C_0(1-e^{-t})} dV_{g_0},\ \ \  t\in [0,T).
\end{equation}
In particular,  $dV_{g_t}$ is uniformly bounded from above for $t\in [0,T)$.
\end{enumerate}

A deeper observation is the following theorem of Zhang \cite{Zhang09} (for a generalization to the semi-ample case, see Song-Tian \cite{ST16}).

\begin{theorem}\label{th:scalar_bdd}
If $M$ is a projective manifold such that $K_M$ is big and nef,  then $S_C(g_t)$ is uniformly bounded for all $t\in [0,\infty)$.
\end{theorem}

\subsection{Singular K\"ahler-Einstein metrics}
The K\"ahler-Ricci flow was first used by Cao \cite{Cao} to give an alternative proof of Yau's  theorem on the existence of K\"ahler-Einstein metric on compact K\"ahler manifolds with trivial or negative first Chern class (cf.  Yau \cite{Yau}; see also Aubin \cite{Aubin} for the negative case).  Tsuji \cite{Tsuji88} (see also Tian-Zhang \cite{TZ06}) proved the following

\begin{theorem}\label{th:Tsuji}
 If $K_M$ is big and nef,  then the solution $g_t$ of \eqref{eq:KR2} converges in $C^\infty_{\rm loc}-$topology on some Zariski open set  $U\subset M$ to a K\"ahler-Einstein metric $g_{KE}$ on $U$.  Moreover,  the K\"ahler form of  $g_{KE}$  extends to a closed positive current on $M$.  
\end{theorem} 
 
One calls   $g_{KE}$ a singular K\"ahler-Einstein metric on $M$.  Tsuji's work was extended by Eyssidieux-Guedj-Zeriahi \cite{EGZ},  Boucksom et al.  \cite{BEGZ10} and Song-Tian \cite{ST12} to general projective manifolds of general type.

\begin{theorem}[cf. \cite{ST12}, Theorem B1, C1]\label{th:ST}
Let $M$ be a projective manifold of general type. Then there exists a measure $dV_{KE}$ such that 
\begin{enumerate}
\item $(K_M,(dV_{KE})^{-1})$ is an analytic Zariski decomposition,
\item $\omega_{KE}:=i\partial\bar{\partial}\log dV_{KE}$ is a closed positive current on $M$; moreover,  $\omega_{KE}$ is smooth on a Zariski open set $U\subset M$ and satisfies
$$
{\rm Ric}(\omega_{KE})=-\omega_{KE} \ \ \ \text{on\ \ }U.
$$
\end{enumerate}
\end{theorem}

\subsection{The Burns-Simanca metric}\label{subsec:BS}

Denote by $\mathrm{Bl}_0\mathbb{C}^n$ the blow-up of $\mathbb{C}^n$ at the origin, i.e.,
\[
\mathrm{Bl}_0\mathbb C^n:=\left\{(z,[\zeta])\in\mathbb{C}^n\times\mathbb{P}^{n-1},\ z_j\zeta_k=z_k\zeta_j,\ 1\leq j,k\leq n\right\}.
\]
Here $[\zeta]=[\zeta_1:\cdots:\zeta_n]$ is the homogeneous coordinate on $\mathbb{P}^{n-1}$. It is known that $\mathrm{Bl}_0\mathbb{C}^n$ is the total space of the tautological line bundle over $\mathbb{P}^{n-1}$ with a natural projection $\pi:\mathrm{Bl}_0\mathbb{C}^n\rightarrow\mathbb{P}^{n-1}$. In particular, $\mathrm{Bl}_0\mathbb C^n$ is a complex manifold. There is another natural holomorphic map
\[
\varpi:\mathrm{Bl}_0\mathbb C^n\rightarrow\mathbb{C}^n,\ \ \ (z,[\zeta])\mapsto z,
\]
which maps $\mathrm{Bl}_0\mathbb C^n\setminus\varpi^{-1}(0)$ biholomorphically to $\mathbb{C}^n\setminus\{0\}$. Recall that $\varpi^{-1}(0)$ is called the exceptional locus, which is a submanifold of $\mathrm{Bl}_0\mathbb C^n$ (biholomorphic to $\mathbb{P}^{n-1}$).

The Burns-Simanca metric $g_{BS}$ is a K\"{a}hler metric on $\mathrm{Bl}_0\mathbb{C}^n$ with zero scalar curvature (cf. \cite{LeBrun88,Simanca91}). On the complement of the exceptional divisor, its K\"{a}hler form can be written as
\[
\omega_{BS}=i\partial\bar{\partial}\left(\frac{|z|^2}{2}+\psi(z)\right),
\]
where
\begin{equation}\label{eq:BS_zero}
\psi(z)=a\log|z|^2+\eta(|z|^2)
\end{equation}
for some smooth function $\eta$ on $[0,+\infty)$ and $a>0$. When $n=2$, one can take $a=1$ and $\eta\equiv0$; when $n\geq3$, $\eta$ cannot be written explicitly, but we have the asymptotic expansion
\begin{equation}\label{eq:BS_infty}
\psi(z)=-|z|^{4-2n}+O(|z|^{2-2n}),\ \ \ |z|\to \infty
\end{equation}
(see \cite{Szekelyhidi12}).

The formula \eqref{eq:BS_zero} enables us to extend $\omega_{BS}$ through the exceptional locus. Indeed, if $(z,[\zeta])\in\mathrm{Bl}_0\mathbb{C}^n$, then
\[
\log|z|^2=\log\left(\sum^n_{j=1}\frac{|\zeta_j|^2|z_k|^2}{|\zeta_k|^2}\right)=\log|\zeta|^2+\log|z_k|^2-\log|\zeta_k|^2,
\]
where $z_k,\zeta_k\neq0$, so that
\[
i\partial\bar{\partial}\log|z|^2=i\partial\bar{\partial}\log|\zeta|^2.
\]
As a consequence, we have
\begin{equation}\label{eq:BS_FS}
\omega_{BS}=a\cdot \pi^*\omega_{FS}+\varpi^*\left(i\partial\bar{\partial}\left(\frac{|z|^2}{2}+\eta(|z|^2)\right)\right)
\end{equation}
outside the exceptional locus, where $\omega_{FS}$ is the Fubini-Study metric on $\mathbb{P}^{n-1}$.

\subsection{An $L^2$ estimate for $\bar{\partial}$}

We present the following refinement of Demailly-Nadel's vanishing theorem which might be useful for other purposes.  

\begin{theorem}\label{th:L2}
Let $M$ be a projective $n-$manifold equipped with a K\"{a}hler metric $g$ and $L$ a holomorphic line bundle over $M$ equipped with a singular Hermitian metric $h=e^{-\varphi}$ satisfying 
\begin{enumerate}
\item  $\Theta_h$ is a closed positive current on $M$,  
\item $h$ is smooth on some open  set $U\subset M$ and there exists  a continuous positive $(1,1)$-form on $\gamma$ on  $U$ such that 
$$
i\partial\bar{\partial}\varphi\geq\gamma\ \ \ \text{on\ \ } U.
$$
\end{enumerate}
Then for any $\bar{\partial}$-closed $v\in L^2_{(0,1)}(M,K_M+ L)$ with $\mathrm{supp}\,v\subset U$ and 
$$
\int_M|v|^2_{(dV_g)^{-1},\gamma}e^{-\varphi}dV_g:=\int_U |v|^2_{(dV_g)^{-1},\gamma}e^{-\varphi}dV_g<+\infty, 
$$
there exists $u\in L^2(M,K_M+ L)$ such that $\bar{\partial}u=v$ and
\[
\int_M|u|^2_{(dV_g)^{-1}}e^{-\varphi}dV_g\leq\int_M |v|^2_{(dV_g)^{-1},\gamma}e^{-\varphi}dV_g.
\]
Here if we write $v=\sum v_j dz_1\wedge\cdots \wedge dz_n \wedge d\bar{z}_j \otimes \xi$ and $\gamma=\sum \gamma_{jk} dz_j\wedge d\bar{z}_k$ in local coordinates,  then 
$$
|v|^2_{(dV_g)^{-1},\gamma} = \sum \gamma^{jk} v_j\bar{v}_k,\ \ \ (\gamma^{jk})=(\gamma_{jk})^{-1}.
$$
\end{theorem}

\begin{proof}
Take an ample divisor $E$ of  $M$, so that $M':=M\setminus E$ is a Stein manifold.  Fix a smooth Hermitian metric $h_0=e^{-\varphi_0}$ of $L$ and set $\psi=\varphi-\varphi_0$.  Then 
\begin{enumerate}
\item
$i\partial\bar{\partial}\psi+i\partial\bar{\partial}\varphi_0\geq0$ on $M$ in the sense of currents,
\item
$i\partial\bar{\partial}\psi+i\partial\bar{\partial}\varphi_0\geq\gamma$ on $U$.
\end{enumerate}
Since $M'$ is Stein,  there exists a  nontrivial holomorphic section $f$ of $L$ over  $M'$,  in view of Cartan's Theorem A.   It follows that 
\[
\phi:=\psi-\log|f|^2_{h_0}
\]
is a smooth plurisubharmonic (psh) exhaustion function on $M'':=M'\setminus\{f=0\}$,  which satisfies 
$$
i\partial\bar{\partial}\phi\geq\gamma\ \ \ \text{on\ \ } M''\cap U.
$$
In particular,  
$M''$ is  a Stein manifold since $M'$ is,  which can be realized as a closed complex submanifold of $\mathbb{C}^{2n+1}$.  Thus there exists a neighbourhood $W$ of $M''$ in $\mathbb C^{2n+1}$ and a holomorphic retraction $\Psi:W\rightarrow M''$, so that $\phi\circ\Psi$ is psh on $W$. Let $\{M_j\}$ be an increasing sequence of relatively compact Stein domains in $M''$ such that $\bigcup^\infty_{j=1}M_j=M''$. By standard regularization of $\phi\circ \Psi$,  we may find a smooth strictly psh function $\phi_j$ on  $M_j$ such that $\phi_j\downarrow\phi$ and
\[
i\partial\bar{\partial}\phi_j\geq\gamma\ \ \ \text{on }M_j\cap U_0.
\]

With $\varphi_j:=\varphi_0+\phi_j+\log|f|^2_{h_0}$,   $h_j:=e^{-\varphi_j}$ becomes a smooth Hermitian metric of $L$ over $M_j$ satisfying
\begin{enumerate}
\item
$i\partial\bar{\partial}\varphi_j\geq i\partial\bar{\partial}\phi_j>0$ on $M_j$,
\item
$i\partial\bar{\partial}\varphi_j\geq\gamma$ on $M_j\cap{U}$,
\item
$\varphi_j\downarrow\varphi$ as $j\uparrow+\infty$.
\end{enumerate}
By standard $L^2$ estimates for $\bar{\partial}$ on complete K\"ahler manifolds (see e.g.,  \cite[Theorem 5.2]{DemaillyBook}), we can find a solution $u_j$ of  $\bar{\partial}u=v$ on $M_j$ such that 
\[
\int_{M_j}|u_j|^2_{(dV_g)^{-1}}e^{-\varphi_j}dV_g
\leq \int_{M_j}|v|^2_{(dV_g)^{-1},i\partial\bar{\partial}\phi_j}e^{-\varphi_j}dV_g
\leq \int_M|v|^2_{(dV_g)^{-1},\gamma}e^{-\varphi}dV_g.
\]
Thanks to the Banach-Alaoglu theorem,  there exists a weak limit $u$ of $\{u_j\}$ such that $\bar{\partial}u=v$ holds in the sense of distributions on $M''$,  together with estimate
\[
\int_{M''}|u|^2_{(dV_g)^{-1}}e^{-\varphi}dV_g
\leq \int_M|v|^2_{(dV_g)^{-1},\gamma}e^{-\varphi}dV_g.
\]
The conclusion follows immediately from the well-known  $L^2$ extension property of the $\bar{\partial}$-equation across analytic sets (cf.  \cite[Lemma 11.10]{DemaillyBook}).
\end{proof}

\section{Upper asymptotic bounds of the Bergman kernel}\label{sec:upper_bound}

\begin{proof}[Proof of Theorem \ref{th:Bergman1}]
For the sake of simplicity,  we write
$$
|s|_m^2:= |s|^2_{(dV_g)^{-\otimes m}}=\frac{|f|^2}{(\det(g_{j\bar{k}}))^m},  \ \ \ s=f\, (dz_1\wedge\cdots\wedge dz_n)^{\otimes m}\in \Gamma(M,mK_M).
$$
Recall that the Bergman kernel function $B_m$ enjoys the following extremal property
\begin{equation}\label{eq:extremal}
B_m(x)=\sup\left\{|s(x)|^2_m:s\in{\Gamma(M,mK_M)},\ \int_M|s|^2_mdV_g\leq1\right\}.
\end{equation}
 A trivial but very useful formula is 
\begin{equation}\label{eq:plurigenera}
P_m(M)=\int_M B_m dV_g.
\end{equation}

Let $\lambda_1,\cdots,\lambda_n$ be the eigenvalues of $Ric(\omega)=i\sum R_{j\bar{k}}dz_j\wedge d\bar{z}_k$ with respect to $g$, that is, the eigenvalues of the matrix $(g_{j\bar{k}})^{-1}(R_{j\bar{k}})$, with $\lambda_1\leq\lambda_2\leq\cdots\leq\lambda_n$. Note that $\lambda_j=\lambda_j(x)$ is merely a continuous function on $M$.  If we write $\varphi=\log\det(g_{j\bar{k}})$,  then
\begin{equation}\label{eq:det}
(i\partial\bar{\partial}\varphi)^n=(-1)^nn!\lambda_1\cdots\lambda_ndV_g
\end{equation}
and
\begin{equation}\label{eq:tr}
\mathrm{Tr}_g(i\partial\bar{\partial}\varphi)=-S_C=-(\lambda_1+\cdots+\lambda_n).
\end{equation}

Fix $x\in{M}$ and take a local coordinate $(z_1,\cdots,z_n)$ around $x$ with $z(x)=0$.  Near $x$, the K\"ahler form $\omega$ of the Hermitian metric $g$ may be expressed as
\[
\omega=\frac{i}{2}\sum^n_{j=1}dz_j\wedge{d\bar{z}_j}+ O(|z|),
\]
so that
\[
dV_g:=\frac{\omega^n}{n!}=\left(1+O(|z|)\right)\left(\frac{i}{2}\right)^n dz_1\wedge{d\bar{z}_1}\wedge\cdots\wedge{dz_n\wedge{d\bar{z}_n}}=:\left(1+O(|z|)\right)dV_0.
\]
Moreover, we have
\[
\varphi(z)=\varphi(0)-\sum^n_{j=1}\lambda_{j}|z_j|^2 + 2\mathrm{Re}\,P(z)+O(|z|^3),
\]
where $P$ is a complex polynomial of degree $2$ with $P(0)=0$.

Take $0<r\ll1$. For every $s\in\Gamma(M,mK_M)$ with $\int_M|s|^2_m e^{-m\varphi}dV_g\leq1$, we can identify $s$ with a holomorphic function near $x$,  so that
\begin{eqnarray*}
1
&\geq& \int_{|z_j|<r,\atop 1\leq{j}\leq{n}}|s|^2e^{-m\varphi}dV_g\\
&=& \int_{|z_j|<r,\atop1\leq{j}\leq{n}}\left|s(z)e^{-mP(z)}\right|^2e^{-m\varphi(0)+m\sum^n_{j=1}\lambda_j|z_j|^2+O(m|z|^3)}\left(1+O(|z|)\right)dV_0\\
&\ge & e^{-m\varphi(x)}\left(1-Cr\right)\left(1-Cmr^3\right)\int_{|z_j|<r,\atop1\leq{j}\leq{n}}\left|s(z)e^{-mP(z)}\right|^2e^{m\sum^n_{j=1}\lambda_j|z_j|^2}dV_0\\
&\geq& |s(x)|^2e^{-m\varphi(x)}\left(1-Cr\right)\left(1-Cmr^3\right)\prod^n_{j=1}\int_{|z_j|<r}e^{m\lambda_j |z_j|^2}dV_{z_j},
\end{eqnarray*}
in view of the mean value inequality, where $dV_{z_j}$ is the Euclidean volume element in the $z_j-$plane.  Here and in what follows, $C$ denotes a generic positive constant and  $r=o(m^{-1/3})$ as $m\rightarrow+\infty$. It follows from \eqref{eq:extremal} that
\begin{equation}\label{eq:Bergman_upper_1a}
B_{m}(x)\leq \left(1+Cr+Cmr^3\right)\prod^n_{j=1}\left(\int_{|z_j|<r}e^{m\lambda_j|z_j|^2}dV_{z_j}\right)^{-1}.
\end{equation}
A straightforward calculation shows 
\begin{eqnarray*}
\int_{|z_j|<r}e^{m\lambda_j|z_j|^2}dV_{z_j}
&=& \begin{cases}
\frac{\pi}{m\lambda_j}(e^{m\lambda_j r^2}-1),\ \ \ &\lambda_j\neq0,\\
\pi{r}^2,\ \ \ &\lambda_j=0.
\end{cases}
\end{eqnarray*}

Now suppose that $ Ric(\omega)|_x=-i\partial\bar{\partial}\varphi|_x$ admits $p=p(x)$ positive eigenvalues and $q=q(x)$ negative eigenvalues, i.e.,
\begin{equation}\label{eq:eigenvalues}
\lambda_1\leq\cdots\leq\lambda_q<0,\ \ \ \lambda_{q+1}=\cdots=\lambda_{n-p}=0,\ \ \ \lambda_n\geq\cdots\geq\lambda_{n-p+1}>0.
\end{equation}
Note that $p(x)$ and $q(x)$ are lower semicontinuous functions on $M$.

Take $A_m>0$ with $\lim_{m\rightarrow+\infty}A_m=+\infty$ and $\lim_{m\rightarrow+\infty}A_m/m^{1/3}=0$. Let us choose $r=(A_m/m)^{1/2}$.
\begin{itemize}
\item[$(1)$]
If $1\leq j\leq q$,  then
\[
\int_{|z_j|<r}e^{m\lambda_j|z_j|^2}dV_{z_j}
=\frac{\pi}{|\lambda_j|}\frac{1-e^{-|\lambda_j|A_m}}{m}\sim \frac{\pi}{|\lambda_j|m},\ \  \text{as\ } m\rightarrow+\infty.
\]

\item[$(2)$]
If $q+1\leq j\leq n-p$,  then
\[
\int_{|z_j|<r}e^{m\lambda_j|z_j|^2}dV_{z_j}
=\pi\cdot\frac{A_m}{m}.
\]

\item[$(3)$]
If $n-p+1\leq j\leq n$,  then
\[
\int_{|z_j|<r}e^{m\lambda_j|z_j|^2}dV_{z_j}
=\frac{\pi}{\lambda_j}\frac{e^{\lambda_jA_m}-1}{m}
\sim\frac{\pi}{\lambda_jm}e^{\lambda_jA_m},\ \  \text{as\ } m\rightarrow+\infty.
\]
\end{itemize}
These together with \eqref{eq:Bergman_upper_1a} yield \eqref{eq:Bergman_upper_1_1}.
\end{proof}

\begin{remark}
Theorem \ref{th:Bergman1} implies
\begin{equation}\label{eq:Bergman_upper_1}
\limsup_{m\rightarrow+\infty}\frac{B_m(x)}{m^n}\leq
(-1)^n\frac{\lambda_1\cdots\lambda_n}{\pi^n},\ \ \ x\in{M_0}
\end{equation}
\begin{equation}\label{eq:Bergman_upper_2}
\limsup_{m\rightarrow+\infty}\frac{B_m(x)}{m^n}=0,\ \ \ x\notin{M_0},
\end{equation}
where $M_0:=\{x\in M:Ric(\omega)|_x<0\}=\{x\in M:q(x)=n\}$. Both \eqref{eq:Bergman_upper_1} and \eqref{eq:Bergman_upper_2} are well-known (see e.g.,  \cite{Berman},  \cite{BerndtssonBook}).
\end{remark}

\begin{proof}[Proof of  Theorem \ref{th:vanishing}]
Choose $A_m=m^\alpha$ for $\alpha\in(0,1/3)$ in Theorem \ref{th:Bergman1}. If $p(x)>0$, i.e., $Ric(\omega)|_x$ admits a positive eigenvalue $\lambda_1$, then
\begin{equation}\label{eq:vanishing_Bergman}
B_m(x)\leq e^{-cm^\alpha}
\end{equation}
for some $c>0$ independent of $x$. If $Ric(\omega)|_x$ admits a positive eigenvalue everywhere on $M$, then \eqref{eq:vanishing_Bergman} holds for every $x\in M$. It follows from \eqref{eq:plurigenera} that $\kappa(M)=-\infty$, i.e., $P_m(M)=0$,  $\forall\,m$.
\end{proof}

\begin{proof}[Proof of  Proposition \ref{prop:RicciRank}]
By the proof of Theorem \ref{th:vanishing},  it suffices to assume  $p(x)=0$, i.e., $Ric(\omega)|_x\leq0$.  With $A_m=m^\alpha$ for $\alpha\in(0,1/3)$, Theorem \ref{th:Bergman1} yields
\begin{equation}\label{eq:RicciRank_Bergman}
B_m(x)\leq \left(1+Cm^{\frac{3\alpha-1}{2}}\right)\frac{1}{\pi^n}\left(\prod_{\lambda_j\neq0}|\lambda_j|\right)m^{\alpha q+(1-\alpha)n}
\end{equation}
for some positive constant $C$ independent of $x$. Note that $q\leq \mathcal{R}_g$. By \eqref{eq:plurigenera}, \eqref{eq:vanishing_Bergman} and \eqref{eq:RicciRank_Bergman}, we have
\[
\kappa(M)\leq \alpha \mathcal{R}_g+(1-\alpha)n=\mathcal{R}_g+(1-\alpha)(n-\mathcal{R}_g),\ \ \ \forall\,0<\alpha<1/3.
\]
The conclusion follows by letting $\alpha\rightarrow1/3$.
\end{proof}

\section{Lower asymptotic bounds of the Bergman kernel}

With the help of Theorem \ref{th:L2},  Tian's method of peak functions (cf.  \cite{Tian}) applies to the following

\begin{theorem}\label{th:BergmanLower}
Let $M$ be a projective $n-$manifold equipped with a K\"ahler metric $g_0$.  Let $g$ be a smooth Hermitian metric defined on an open set $U\subset M$ with K\"{a}hler form $\omega$,  such that
\begin{enumerate}
\item $Ric(\omega)<0$ on $U$,
\item the curvature of the metric $(dV_g)^{-1}$ on $K_M|_{U}$ extends to  a closed positive current on $M$. 
\end{enumerate}
Let $\lambda_1,\cdots,\lambda_n$ be the eigenvalues of $Ric(\omega)$ with respect to $g$ on $U$.  Set 
$$
\widetilde{B}_m(x) = B_{g_0,(dV_g)^{-1},mK_M}(x).  
$$
Then 
\begin{equation}\label{eq:AsympLower}
\liminf_{m\rightarrow+\infty}\frac{\widetilde{B}_m(x)}{m^n}\geq (-1)^n\frac{\lambda_1\cdots\lambda_n}{\pi^n}\cdot \frac{dV_g}{dV_{g_0}}(x),
\ \ \ \forall\ x\in U.
\end{equation}
\end{theorem}

\begin{proof}
For any fixed $x\in{M}$, we can take a local coordinate $(z_1,\cdots,z_n)$ around $x$ such that $z(x)=0$,  and the K\"ahler form $\omega_0$ of $g_0$ may be written as
\[
\omega_0=\sum dz_j\wedge{d\bar{z}_j}+ O(|z|).
\]
Moreover,  if one writes $(dV_g)^{-1}=e^{-\varphi}$ near $x$,  then 
\[
\varphi(z)=\varphi(0)-\sum \lambda_{j}|z_j|^2 + 2\mathrm{Re}\,P(z)+O(|z|^3),
\]
where $P$ is a complex polynomial of degree $2$ with $P(0)=0$.   Without loss of generality,  we may assume that the coordinate neighbourhood is the unit ball $\mathbb B^n$ and it is contained in $U$; moreover,   $Ric(\omega)\leq -2\varepsilon \omega$ on $\mathbb{B}^n$ for some $\varepsilon>0$,  so that $\lambda_j\leq -2\varepsilon$ for $j=1,2,\cdots,n$.

As before, take $A_m>0$ with $A_m\rightarrow+\infty$ and $A_m/m^{1/3}\rightarrow0$ as $m\rightarrow+\infty$. Let $\chi:[0,+\infty)\rightarrow[0,1]$ be a smooth cut-off function such that $\chi|_{[0,1]}=1$ and $\chi|_{[2,+\infty)}=0$.   Set
\begin{eqnarray*}
\chi_m(z) & := & \chi\left(\left(m/A_m\right)^{1/2}z\right)e^{mP(z)} (dz_1\wedge\cdots\wedge dz_n)^{\otimes (m+1)}\\
v_m & := & e^{mP}\bar{\partial}\chi_m\otimes  (dz_1\wedge\cdots\wedge dz_n)^{\otimes (m+1)}.
\end{eqnarray*}
Note that
\begin{eqnarray}\label{eq:chi_m_L2}
\int_M|\chi_m|^2_{(dV_{g_0})^{-1}\otimes (dV_g)^{\otimes (-m)}} dV_{g_0} 
&\le & \int_{|z|\leq 2(A_m/m)^{1/2}} e^{-m\varphi(0)+m\sum^n_{j=1}\lambda_j|z_j|^2-O(m|z|^3)}(1+O(|z|))dV_0(z)\notag\\
&\le& e^{-m\varphi(x)}\left(1-C\frac{A_m^{3/2}}{m^{1/2}}\right)\int_{|z|\leq 2(A_m/m)^{1/2}} e^{m\sum^n_{j=1}\lambda_j|z_j|^2}dV_0(z)\notag\\
&=& \frac{e^{-m\varphi(x)}}{m^n}\left(1-C\frac{A_m^{3/2}}{m^{1/2}}\right)\int_{|\zeta|\leq 2 A_m^{1/2}} e^{\sum^n_{j=1}\lambda_j|\zeta_j|^2}dV_0(\zeta)\notag\\
&\sim& \frac{e^{-m\varphi(x)}}{m^n}\int_{\zeta\in\mathbb{C}^n} e^{\sum^n_{j=1}\lambda_j|\zeta_j|^2}dV_0(\zeta)\notag\\
&=& \frac{\pi^n}{m^n}\left(\prod^n_{j=1}|\lambda_j|\right)^{-1}e^{-m\varphi(x)},\ \ \ m\rightarrow+\infty.
\end{eqnarray}
Moreover,
\begin{eqnarray}\label{eq:v_m_L2}
&& \int_M|v_m|^2_{(dV_{g_0})^{-1}\otimes (dV_g)^{\otimes (-m)}} dV_{g_0}\nonumber\\
&\leq&  \frac{m}{A_m}\sup|\chi'|^2\int_{(A_m/m)^{1/2}\leq |z|\leq 2(A_m/m)^{1/2}}e^{-m\varphi(0)+m\sum^n_{j=1}\lambda_j|z_j|^2-O(m|z|^3)}(1+O(|z|))dV_0(z)\notag\\
&\leq& \frac{m e^{-m\varphi(x)}}{A_m}\sup|\chi'|^2 \left(1+C\frac{A_m^{3/2}}{m^{1/2}}\right) \int_{(A_m/m)^{1/2}\leq |z|\leq 2(A_m/m)^{1/2}}e^{-m\varepsilon|z|^2}dV_0(z)\notag\\
&=& \frac{e^{-m\varphi(x)}}{m^{n-1}A_m}\sup|\chi'|^2\left(1+C\frac{A_m^{3/2}}{m^{1/2}}\right)\int_{A_m^{1/2}\leq |\zeta|\leq 2A_m^{1/2}}e^{-\varepsilon|\zeta|^2}dV_0(\zeta)\notag\\
&\leq& C_0\frac{A_m^{n-1}e^{-\varepsilon A_m}}{m^{n-1}}e^{-m\varphi(x)},\ \ \ \forall\,m\gg1,
\end{eqnarray}
where $C_0$ is a numerical constant.

Let us take $A_m$ such that $e^{-\varepsilon A_m}= o(m^{-N})$ for all $N>0$ (e.g., $A_m=(\log m)^2$). By Theorem \ref{th:L2} and \eqref{eq:v_m_L2}, there exists a solution $u_m$ of $\bar{\partial}u_m=v_m$ on $M$,  such that 
\begin{eqnarray}\label{eq:u_m_L2}
\int_M|u_m|^2_{(dV_{g_0})^{-1}\otimes (dV_g)^{\otimes (-m)}} dV_{g_0}
&\leq& \int_M|v_m|^2_{(dV_{g_0})^{-1}\otimes (dV_g)^{\otimes (-m)},-mRic(\omega)} dV_{g_0}\notag\\
&\leq& \frac{1}{2\varepsilon m}\int_M|v_m|^2_{(dV_{g_0})^{-1}\otimes (dV_g)^{\otimes (-m)}} dV_{g_0}\notag\\
&\leq& \frac{C_0}{\varepsilon}\frac{A_m^{n-1}e^{-\varepsilon A_m}}{m^n}e^{-m\varphi(x)}\notag\\
&\leq& \frac{C_0}{\varepsilon}\frac{1}{m^{n+1}}e^{-m\varphi(x)},\ \ \ \forall\,m\gg1.
\end{eqnarray}

Let $u_m^\ast$ be a local representation of  $u_m$ at $x$.    It follows that
\begin{eqnarray*}
& & \int_{|z|<(A_m/m)^{1/2}}|u_m|^2_{(dV_{g_0})^{-1}\otimes (dV_g)^{\otimes (-m)}} dV_{g_0}\\
&=& \int_{|z|<(A_m/m)^{1/2}}\left|u_m^\ast (z)e^{-mP(z)}\right|^2e^{-m\varphi(0)+m\sum^n_{j=1}\lambda_j|z_j|^2+O(m|z|^3)}\left(1+O(|z|)\right)dV_0(z)\\
&\geq& e^{-m\varphi(x)}\left(1-C\frac{A_m^{3/2}}{m^{1/2}}\right)\int_{|z|<(A_m/m)^{1/2}}\left|u_m^\ast(z)e^{-mP(z)}\right|^2e^{m\sum^n_{j=1}\lambda_j|z_j|^2}dV_0(z)\\
&\geq& |u_m^\ast (x)|^2e^{-m\varphi(x)}\left(1-C\frac{A_m^{3/2}}{m^{1/2}}\right)\int_{|z|<(A_m/m)^{1/2}}e^{m\sum^n_{j=1}\lambda_j |z_j|^2}dV_0(z)\\
&=& |u_m^\ast (x)|^2e^{-m\varphi(x)}\frac{1}{m^n}\left(1-C\frac{A_m^{3/2}}{m^{1/2}}\right)\int_{|\zeta|<A_m^{1/2}}e^{\sum^n_{j=1}\lambda_j |\zeta_j|^2}dV_0(\zeta)
\end{eqnarray*}
in view of the mean value inequality. This combined with \eqref{eq:u_m_L2} yields
\begin{eqnarray}\label{eq:u_m_x}
|u_m^\ast(x)|^2e^{-m\varphi(x)}
&\leq& \frac{C_0}{\varepsilon} \frac{1}{m} \left(1-C\frac{A_m^{3/2}}{m^{1/2}}\right)^{-1}\left(\int_{|\zeta|<A_m^{1/2}}e^{\sum^n_{j=1}\lambda_j |\zeta_j|^2}dV_0(\zeta)\right)^{-1}e^{-m\varphi(x)}\notag\\
&\sim& \frac{C_0}{\varepsilon} \frac{1}{m}\left(\int_{\zeta\in\mathbb{C}^n}e^{\sum^n_{j=1}\lambda_j |\zeta_j|^2}dV_0(\zeta)\right)^{-1}e^{-m\varphi(x)}\notag\\
&\leq& \frac{C}{m}e^{-m\varphi(x)}.
\end{eqnarray}
Recall that $C$ denotes a generic constant depending on the geometry of $(M,g)$, but is independent of $m$.

Set $s_m:=\chi_m-u_m$.  Clearly, $s_m\in \Gamma(M,(m+1)K_M)$. It follows from \eqref{eq:chi_m_L2}, \eqref{eq:u_m_L2} and \eqref{eq:u_m_x} that
\begin{eqnarray*}
|s_m^\ast (x)|^2e^{-m\varphi(x)}
&\geq& \left(|\chi_m^\ast (x)|e^{-m\varphi(x)/2}-|u_m^\ast (x)|e^{-m\varphi(x)/2}\right)^2\\
&\geq& \left(1-\frac{C^{1/2}}{m^{1/2}}\right)^2e^{-m\varphi(x)}\\
&\sim& e^{-m\varphi(x)},\ \ \ m\rightarrow+\infty,
\end{eqnarray*}
and
\begin{eqnarray*}
&& \int_M|s_m|^2_{(dV_{g_0})^{-1}\otimes (dV_g)^{\otimes (-m)}} dV_{g_0}\\
&\leq& \left(\left(\int_M|\chi_m|^2_{(dV_{g_0})^{-1}\otimes (dV_g)^{\otimes (-m)}} dV_{g_0}\right)^{1/2}+\left(\int_M|u_m|^2_{(dV_{g_0})^{-1}\otimes (dV_g)^{\otimes (-m)}} dV_{g_0}\right)^{1/2}\right)^2\\
&\sim& \frac{\pi^n}{m^n}\left(\prod^n_{j=1}|\lambda_j|\right)^{-1}e^{-m\varphi(x)},\ \ \ m\rightarrow+\infty.
\end{eqnarray*}
These together with the extremal property \eqref{eq:extremal} yield \eqref{eq:AsympLower}.  
\end{proof}

\section{Proof of Theorem \ref{th:Main0}}
Recall that $M_0$ is the open subset of $M$ where $\lambda_j<0$ for all $1\leq{j}\leq{n}$. By \eqref{eq:plurigenera},  \eqref{eq:Bergman_upper_1}, \eqref{eq:Bergman_upper_2} and Fatou's lemma,  we have

\begin{eqnarray}\label{eq:LeBrun_M_0}
\mathrm{CanVol}(M) & = & \limsup_{m\rightarrow+\infty}\frac{P_m(M)}{m^n/n!}\leq \int_M\limsup_{m\rightarrow+\infty}\frac{B_m}{m^n/n!}dV_g\nonumber\\
&\leq& \frac{(-1)^nn!}{\pi^n}\int_{M_0}\lambda_1\cdots\lambda_n{dV_g}\nonumber\\
&\leq& \frac{n!}{\pi^n}\int_{M_0}\left(\frac{-\lambda_1-\cdots-\lambda_n}{n}\right)^n{dV_g}\nonumber\\
&\leq& \frac{n!}{(n\pi)^n}\int_M|S^-_{C}|^ndV_g,
\end{eqnarray}
i.e.,  \eqref{eq:LeBrun_1} holds,  which  implies \eqref{eq:volume_lower_1} since
$$
{\rm vol}_g (M)\ge \int_M|S^-_{C}(g)|^ndV_g
$$
if $S_C(g)\ge -1$.

Suppose there exists a Hermitian metric $g$ such that $S_C(g)\ge -1$ and
$$
{\rm vol}_g(M) = \frac{(n\pi)^n}{n!} {\rm CanVol}(M).
$$ 
In particular,  ${\rm CanVol}(M)>0$,  i.e.,  $M$ is of general type.  
On the other hand,  \eqref{eq:LeBrun_M_0} implies 
$$
{\rm vol}_g(M_0) \ge \frac{(n\pi)^n}{n!} {\rm CanVol}(M).
$$ 
Thus
$$
{\rm vol}_g(M\setminus M_0)=0
$$
and equalities in \eqref{eq:LeBrun_M_0} hold. It follows that $\lambda_1=\cdots=\lambda_n=:\lambda$ on $M_0$, hence on $M$ by continuity. Note that
\[
\lambda=\frac{1}{n}\mathrm{tr}_g(\lambda\omega)=\frac{1}{n}\mathrm{tr}_g(Ric(\omega))=\frac{S_C(g)}{n}.
\]
In case $g$ is K\"{a}hler, the Chern connection coincides with the Levi-Civita connection, and hence $g$ is also a Riemannian Einstein metric. Since the real dimension of $M$ is greater than $2$, we conclude that $\lambda$ is a constant (cf. Berger \cite[Theorem 277]{BergerBook}). Moreover, $\lambda<0$ for $M$ is of general type.

\section{Proofs  of Theorem \ref{th:Main1} and Corollary \ref{cor:holo_increasing}}

\subsection{Proof of Theorem \ref{th:Main1} when $K_M$ is big and nef}
First of all,  we have
\begin{equation}\label{eq:MinVol_S}
\mathrm{MinVol}_C(M)\geq\mathcal{I}_C^-(M).
\end{equation}
Indeed, given any Hermitian metric $g$ on $M$, the scalar curvature $S_C(g)$ is bounded from below on $M$. Since $S_C(tg)=t^{-1}S_C(g)$ and $dV_{tg}=t^ndV_g$, we see that the integral $\int_M|S_C^-(g)|^ndV_g$ is invariant under scaling. Then we may assume that $S_C(g)\geq-1$, so that
\[
\int_M|S_C^-(g)|^ndV_g\leq\int_M1dV_g\leq\mathrm{vol}_g(M),
\]
from which \eqref{eq:MinVol_S} immediately follows. 

By \eqref{eq:LeBrun_1} and \eqref{eq:MinVol_S},  it suffices to verify the follow inequalities:
\begin{equation}\label{eq:MinVol_CanVol}
\mathrm{MinVol}_C(M)
\leq \frac{(n\pi)^n}{n!}\mathrm{CanVol}(M)
\end{equation}
\begin{equation}\label{eq:big_nef_2}
\mathcal{I}_C(M)\leq\frac{(n\pi)^n}{n!}\mathrm{CanVol}(M).
\end{equation}

Note that a compact K\"{a}hler manifold with a big line bundle is always algebraic, in view of a celebrated result of Siu \cite{Siu84}.   Thus by Theorem \ref{th:Tsuji},  we know that there exists an analytic subset $E\subset M$  such that  the solution $g_t$ of the normalized K\"ahler-Ricci flow \eqref{eq:KR2} converges in $C^\infty_{\rm loc}(M\setminus E)$ to the singular K\"ahler-Einstein metric $g_{KE}$.  Moreover,  the K\"ahler form $\omega_{KE}$ of  $g_{KE}$  extends to a closed positive current on $M$.   
Consider the Bergman space of $mK_M$ associated to $dV_{KE}$,  i.e.,  
$$
A^2_{g_0,(dV_{KE})^{-1}}(M,mK_M)
$$
and the corresponding Bergman kernel
$$
B_{g_0,(dV_{KE})^{-1},mK_M}(x).
$$
Since
\[
Ric(\omega_{KE})=-\omega_{KE},\qquad \text{on}\ M\setminus E,
\]
it follows from \eqref{eq:AsympLower} that
\begin{equation}\label{eq:limit_Bergman_M_E}
\liminf_{m\rightarrow+\infty}\frac{B_{g_0,(dV_{KE})^{-1},mK_M}(x)}{m^n}\ge  \frac{1}{\pi^n}\cdot \frac{dV_{KE}}{dV_{g_0}}(x),\ \ \ \forall\ x\in M\setminus E.
\end{equation}
Then we have
\begin{eqnarray}\label{eq:lower_CanVol_1}
&& \liminf_{m\rightarrow+\infty}\frac{\dim A^2_{g_0,(dV_{KE})^{-1}}(M,mK_M)}{m^n/n!}\nonumber\\
&=& n!\liminf_{m\rightarrow+\infty}\int_{M}\frac{B_{g_0,(dV_{KE})^{-1},mK_M}(x)}{m^n}dV_{g_0}(x)\notag\\
&\geq& \frac{n!}{\pi^n}\mathrm{vol}_{KE}(M\setminus E)\ \ \ (\text{Fatou's lemma and\ }\eqref{eq:limit_Bergman_M_E}).
\end{eqnarray}
Thus 
\begin{eqnarray}\label{eq:lower_CanVol_2}
\mathrm{CanVol}(M)
= \limsup_{m\rightarrow+\infty}\frac{P_m(M)}{m^n/n!}
& \geq & \limsup_{m\rightarrow+\infty}\frac{\dim A^2_{g_0,(dV_{KE})^{-1}}(M,mK_M)}{m^n/n!}\nonumber\\
& \geq & \frac{n!}{\pi^n}\mathrm{vol}_{KE}(M\setminus E).
\end{eqnarray}

On the other side,   it follows from \eqref{eq:volumeCompare} that there exists for each $\delta>0$, a neighbourhood $E_\delta$ of $E$ such that $\mathrm{vol}_{g_t}(E_\delta)<\delta$ for $t\in[0,\infty)$.  Since $g_t\rightarrow g_{KE}$ in $C^\infty(M\setminus E_\delta)$,  we have $\mathrm{vol}_{g_t} (M\setminus E_\delta)\leq\mathrm{vol}_{KE}(M\setminus E_\delta)+\delta$ when $t\gg1$,  so that
\begin{equation}\label{eq:lower_CanVol_3}
\mathrm{vol}_{g_t} (M)
\leq \mathrm{vol}_{g_t}(M\setminus E_\delta)+\delta
\leq \mathrm{vol}_{KE}(M\setminus E_\delta)+2\delta
\leq \frac{\pi^n}{n!}\mathrm{CanVol}(M)+2\delta,
\end{equation}
in view of \eqref{eq:lower_CanVol_2}.

By \eqref{eq:ScalarLower},  the K\"{a}hler metric $(n+C_0 e^{-t})g_t$ on $M$ satisfies $S_C\geq-1$.  Thus
\begin{equation}\label{eq:MinVolUpper}
\mathrm{MinVol}_C(M)
\leq (n+C_0e^{-t})^n\mathrm{vol}_{g_t}(M).
\end{equation}
This combined with \eqref{eq:lower_CanVol_3} yields
\[
\mathrm{MinVol}_C(M)
\leq (n+C_0 e^{-t})^n\left(\frac{\pi^n}{n!}\mathrm{CanVol}(M)+2\delta\right),  
\]
when $t\gg1$ and $0<\delta\ll1$.  Letting $t\rightarrow\infty$ and $\delta\rightarrow0$, we get \eqref{eq:MinVol_CanVol}. 

For \eqref{eq:big_nef_2},  we infer from 
 Theorem \ref{th:scalar_bdd} that there exists a constant $C>0$ such that 
\begin{equation}\label{eq:scalar_bdd_upper}
|S_C(g_t)|\leq C,\ \ \ \forall\ t\in[0,\infty),
\end{equation}
so that
\begin{equation}\label{eq:I_M_small}
\int_{E_\delta}|S_C(g_t)|^ndV_{g_t} \leq C^n\mathrm{vol}_{g_t} (E_\delta)<C^n\delta.
\end{equation}
Moreover, $S_C(g_t)$ converges uniformly on $M\setminus E_\delta$ to $S_C(g_{KE})=-n$,  so that
\begin{equation}\label{eq:I_M_big}
\lim_{t\rightarrow\infty}\int_{M\setminus E_\delta}|S_C(g_t)|^ndV_{g_t}  = n^n\mathrm{vol}_{KE}(M\setminus E_\delta)\leq \frac{(n\pi)^n}{n!}\mathrm{CanVol}(M),
\end{equation}
in view of \eqref{eq:lower_CanVol_2}. By \eqref{eq:I_M_small} and \eqref{eq:I_M_big}, we have
\begin{equation}\label{eq:I_C_limit}
\mathcal{I}_C (M)\leq\limsup_{t\rightarrow\infty}\int_M|S_C(g_t)|^ndV_{g_t} \leq \frac{(n\pi)^n}{n!}\mathrm{CanVol}(M)+C^n\delta,
\end{equation}
from which  \eqref{eq:big_nef_2} immediately follows.

\subsection{Proof of Theorem \ref{th:Main1} when $K_M$ is nef but not big}
 Recall that the numerical dimension of  the nef line bundle $K_M$ is defined by
\begin{equation}\label{eq:nd}
\mathrm{nd}(M)=\mathrm{nd}(K_M)=\max\left\{k:\ 1\leq k\leq n, c_1(K_M)^k\neq0\ \text{in}\ H^{k,k}(M,\mathbb{R})\right\}.
\end{equation}
Take a smooth $(1,1)$ form $\gamma$ with $c_1(K_M)=[\gamma]$ and  a smooth volume form $\Omega$ on $M$ such that $\gamma=i\partial\bar{\partial}\log\Omega$ and $\int_M\Omega=1$.  Let $g_t$ be the solution to the normalized K\"{a}hler-Ricci flow \eqref{eq:KR2} and $\omega_t$ the associated K\"{a}hler form. Set
\[
\widehat{\omega}_t:=(1-e^{-t})\gamma+e^{-t}\omega_0.
\]
It follows from \eqref{eq:KR2} that
\[
\frac{d}{dt}[\omega_t]=[-\mathrm{Ric}(\omega_t)-\omega_t]=c_1(K_M)-[\omega_t],
\]
i.e.,
\[
\frac{d}{dt}(e^t[\omega_t])=c_1(K_M)e^t.
\]
Solving this ODE in $H^{1,1}(M,\mathbb{R})$, we obtain
\[
[\omega_t]=(1-e^{-t})[c_1(K_M)]+e^{-t}[\omega_0]=[\widehat{\omega}_t].
\]
In particular,
\begin{equation}\label{eq:volume_int_cohomology}
\mathrm{vol}_{g_t}(M)=\int_M\widehat{\omega}_t^n.
\end{equation}
By \eqref{eq:ScalarLower}, we have
\begin{equation}\label{eq:vol_nef_not_big}
\mathrm{MinVol}_C(M)\leq \left(n+C_0e^{-t}\right)^n\mathrm{vol}_{g_t}(M)=\left(n+C_0e^{-t}\right)^n\int_M\widehat{\omega}_t^n.
\end{equation}
Thus it remains to verify $\int_M\widehat{\omega}_t^n\rightarrow0$ as $t\rightarrow+\infty$. This follows immediately from  \cite[Lemma 9.1]{GPSS}. For the sake of completeness, we shall include the very simple proof as follows.  

Since $[\gamma]^k=c_1(K_M)^k=0$ in $H^{k,k}(M,\mathbb{R})$ for $k>\mathrm{nd}(M)$, so
\begin{eqnarray*}
\int_M\widehat{\omega}_t^n
&=& \int_M\left(\gamma + e^{-t}(\omega_0-\gamma)\right)^n\\
&=& \sum^n_{l=0}\binom{n}{l}e^{-(n-l)t}\int_M\gamma^l\wedge(\omega_0-\gamma)^{n-l}\\
&=& e^{-(n-\mathrm{nd}(M))t}\binom{n}{\mathrm{nd}(M)}\int_M\gamma^{\mathrm{nd}(M)}\wedge\left(\omega_0-\gamma\right)^{n-\mathrm{nd}(M)}+o(e^{-(n-\mathrm{nd}(M))t}),\ \ \ t\rightarrow+\infty.
\end{eqnarray*}
In particular, we get
\begin{equation}\label{eq:volume_numerical_dim_1}
C^{-1}e^{-(n-\mathrm{nd}(M))t}\mathrm{vol}_{g_0}(M)
\leq \mathrm{vol}_{g_t}(M)
\leq Ce^{-(n-\mathrm{nd}(M))t}\mathrm{vol}_{g_0}(M).
\end{equation}
for some constant $C$ independent of $t$. As $\mathrm{Kod}(M)<n$ implies $\mathrm{nd}(M)<n$ (cf. \cite{DemaillyBook}), it follows that $\mathrm{vol}_{g_t}(M)\rightarrow0$ as $t\rightarrow+\infty$.

\vspace{2mm}

Finally,  we will show  $\mathcal{I}_C(M)=0$.  Consider the average Chern scalar curvature of $g_t$
\begin{equation}\label{eq:average_sc}
\overline{S}_t
:=\frac{1}{\mathrm{vol}_{g_t}(M)}\int_M S_C(g_t)\,dV_{g_t}.
\end{equation}
We claim that the function $t\mapsto\overline{S}_t$ is bounded on $[0,\infty)$. Since $g_t$ is a K\"ahler metric, we have the standard identity
\[
S_C(g_t)\,\omega_t^n
= n\,Ric(\omega_t)\wedge \omega_t^{n-1},
\]
so that
\[
\int_M S_C(g_t)\,dV_{g_t}
=n\int_M Ric(\omega_t)\wedge \omega_t^{n-1}.
\]
Recall that $[Ric(\omega_t)]=-c_1(K_M)=-[\gamma]$, so
\begin{equation}\label{eq:scalar_integral}
\int_M S_C(g_t)\,dV_{g_t}
= -n\int_M\gamma\wedge\omega_t^{n-1}
= -n\int_M\gamma\wedge\widehat{\omega}_t^{n-1}.
\end{equation}
A similar argument as above yields
\begin{eqnarray*}
\int_M\gamma\wedge\widehat{\omega}_t^{n-1}
&=& \int_M\gamma\wedge\left(\gamma + e^{-t}(\omega_0-\gamma)\right)^{n-1}\\
&=& \sum^{n-1}_{l=0}\binom{n-1}{l}e^{-(n-1-l)t}\int_M\gamma^{l+1}\wedge(\omega_0-\gamma)^{n-1-l}.
\end{eqnarray*}
Note that $[\gamma]^{l+1}=c_1(K_M)^{l+1}=0$ in $H^{l+1,l+1}(M,\mathrm{R})$ when $l\geq\mathrm{nd}(M)$. We get
\[
\left|\int_M\gamma\wedge\widehat{\omega}_t^{n-1}\right|=O(e^{-(n-\mathrm{nd}(M))t}).
\]
This combined with \eqref{eq:average_sc}, \eqref{eq:volume_numerical_dim_1} and \eqref{eq:scalar_integral} gives
\begin{equation}\label{eq:average_sc_bdd}
|\overline{S}_t|\leq C,
\qquad \forall\,t\geq 0.
\end{equation}
Here $C$ is independent of $t$.

We will proceed the argument through suitable conformal deformation. Set
\[
h_t=e^{u_t/n}g_t
\]
for a real-valued smooth function $u_t$ on $M$. The volume form is given by
\begin{equation}\label{eq:volume_form_conformal_t}
dV_{h_t}=e^{u_t}dV_{g_t},
\end{equation}
and a straightforward calculation shows
\begin{equation}\label{eq:sc_conformal_t}
S_C(h_t)=e^{-u_t/n}\left(S_C(g_t)-\Delta_{g_t} u_t\right),
\end{equation}
where $\Delta_{g_t}$ denotes the complex Laplacian with respect to the metric $g_t$, i.e.,
\[
\Delta_{g_t}\varphi=\mathrm{tr}_{g_t}(i\partial\bar{\partial}\varphi),\ \ \ \forall\,\varphi\in C^\infty(M).
\]
For each $t\geq 0$, the Poisson equation
\begin{equation}\label{eq:Poisson}
\Delta_{g_t}u_t = S_C(g_t)-\overline{S}_t
\end{equation}
admits a smooth solution $u_t$ since the right-hand side has integral zero  (see \cite[Theorem 2.12]{SzekelyhidiBook}).  By \eqref{eq:sc_conformal_t} and \eqref{eq:Poisson}, we get
\[
S_C(h_t)=e^{-u_t/n}\overline{S}_t.
\]
This together with \eqref{eq:volume_form_conformal_t} and \eqref{eq:average_sc_bdd} yield
\[
\int_M|S_C(h_t)|^ndV_{h_t}=\int_M|\overline{S}_t|^ndV_{g_t}\leq C^n\mathrm{vol}_{g_t}(M).
\]
Since $\mathrm{nd}(M)<n$, we conclude
\[
\int_M|S_C(h_t)|^ndV_{h_t}\rightarrow0,\ \ \ t\rightarrow+\infty,
\]
in view of \eqref{eq:volume_numerical_dim_1}.
This implies $\mathcal I_C(M)=0$. 

\subsection{Proof of Corollary \ref{cor:holo_increasing}}
The pull-back $F^*K_M$ of $K_M$ is a line bundle over $M'$, whose volume can be defined in a similar way:
\[
\mathrm{vol}(F^*K_M):= \limsup_{m\rightarrow+\infty}\frac{\dim\Gamma(M,F^*K_M)}{m^n/n!}.
\]
It follows from \cite[Lemma 4.3]{Holschbach10} (see also \cite[Lemma 2.9]{Cutkosky24}) that
\begin{equation}\label{eq:deg_vol_pullback}
\mathrm{vol}(F^*K_M)=\deg(F)\cdot\mathrm{CanVol}(M).
\end{equation}
Let  $R_F$ denote the line bundle associated to the effective divisor locally given by the zeros of the complex Jacobian determinant of $F$. We have
\[
K_{M'}=F^*K_M+ R_F.
\]
By \cite[Example 2.2.48]{LazarsfeldBook}, we know that
\begin{equation}\label{eq:volume_increasing_effective}
\mathrm{CanVol}(M')=\mathrm{vol}(F^*K_M+ R_F)\geq\mathrm{vol}(F^*K_M).
\end{equation}
Thus
\[
\mathrm{CanVol}(M')\geq \deg(F)\cdot\mathrm{CanVol}(M).
\]
This together with Corollary \ref{cor:lower_gap} and Theorem \ref{th:Main1} yield \eqref{eq:deg_1}-\eqref{eq:deg_3}. Moreover, if $F$ is a covering map, then $R_F=0$ and $K_{M'}$ is also nef, so that
\[
\mathrm{CanVol}(M')=\mathrm{vol}(F^*K_M)=\deg(F)\cdot\mathrm{CanVol}(M).
\]
Thus equalities in \eqref{eq:deg_1}-\eqref{eq:deg_3} hold, in view of Theorem \ref{th:Main1}.

\begin{remark}
Conversely, if equalities in \eqref{eq:deg_1}-\eqref{eq:deg_3} hold, then $F$ is not necessarily a covering map. A counterexample is given by Theorem \ref{th:blow_up}.
\end{remark}

\section{Proof of Theorem \ref{th:blow_up}}\label{sec:blow_up}

Since the canonical volume is a bimeromorphic invariant (see, e.g., \cite[Proposition 2.2.43]{LazarsfeldBook}, \cite[\S\,2]{Ueno}), it follows that
\begin{equation}\label{eq:canvol_bimeromorphic}
\mathrm{CanVol}(\widehat M)=\mathrm{CanVol}(M).
\end{equation}
By \eqref{eq:canvol_bimeromorphic} and Theorem \ref{th:Main1}, it suffices to prove that $\mathrm{MinVol}_C$ and $\mathcal{I}_C^-$ (resp. $\mathcal{I}_C$) are invariant under the blow-up at finite points when $K_M$ is nef (resp. big and nef). For simplicity,  we will only treat   the  case of blowing-up at one point and leave the general case  to interested readers.  Let
\[
\varpi:\widehat{M}=\mathrm{Bl}_pM\rightarrow M
\]
be the blow-up at $p\in M$.

From  the proof of Theorem \ref{th:Main1}, we learn that if $K_M$ is nef,  then for any $\varepsilon>0$, there exists a K\"{a}hler metric $g$ on $M$ such that $S_C(g)\geq-1$ and
\begin{equation}\label{eq:approximating_metric_I_minus}
\int_M|S_C^-(g)|^ndV\leq \frac{(n\pi)^n}{n!}\mathrm{CanVol}(M)+\varepsilon,
\end{equation}
\begin{equation}\label{eq:approximating_metric_I}
\int_M|S_C(g)|^ndV\leq \frac{(n\pi)^n}{n!}\mathrm{CanVol}(M)+\varepsilon,
\end{equation}
\begin{equation}\label{eq:approximating_metric_vol}
\mathrm{vol}_g(M)\leq \frac{(n\pi)^n}{n!}\mathrm{CanVol}(M)+\varepsilon.
\end{equation}

The main idea is to  modify  $g$ to obtain certain metric on $\widehat{M}$ with analogous properties.  More precisely, we will divide the proof into four steps:
\begin{itemize}
\item[Step 1:] Glue any given metric with the Burns-Simanca metric near the exceptional divisor;
\item[Step 2:] Derive quantitative estimates of the scalar curvature and the volume  for the glued metric;
\item[Step 3:] Deduce the equalities for $\mathcal{I}_C$ and $\mathcal{I}_C^-$;
\item[Step 4:] Apply a conformal correction to achieve the sharp lower bound $S_C\ge -1$, and conclude the equality for $\mathrm{MinVol}_C$.
\end{itemize}

\medskip
\noindent\textbf{Step 1: Construction of the glued metric.}
We will glue the background metric $g$ on $M$ with a rescaled
Burns-Simanca metric near the exceptional divisor,   following
LeBrun \cite{LeBrun96} and Sz{\'e}kelyhidi \cite{Szekelyhidi12}.

Since $g$ is K\"{a}hler, the K\"ahler form $\omega$ of $g$ can be written as
\[
\omega=i\partial\bar{\partial}\left(\frac{|z|^2}{2}+\varphi(z)\right)
\]
in a small coordinate ball $B_\rho$ around $p$, 
where $\varphi(z)=O(|z|^4)$ as $z\rightarrow0$. Let $\delta>0$ be a small parameter. Consider the rescaled Burns-Simanca metric (see \S\,\ref{subsec:BS})
\[
\omega^\delta_{BS}
:=\varpi^*i\partial\bar{\partial}\left(\frac{|z|^2}{2}+\delta^2\psi\left(\frac{z}{\delta}\right)\right)
\]
on $\varpi^{-1}(B_\rho)$, where $z\in B_\rho\setminus\{0\}$. We will define a modified metric $g_\delta$ on $\widehat{M}$ by gluing $g$ and $g_{BS}^\delta$ as follows. Let $\chi:[0,\infty)\to[0,1]$ be a smooth cut-off function with $\chi\equiv 0$ on $[0,1]$ and $\chi\equiv 1$ on $[2,\infty)$.   Take
\[
r_\delta:=\delta^\alpha,\qquad 0<\alpha<{1}/{2}.
\]
Clearly,  $B_{2r_\delta}\subset\subset B_\rho$ for $\delta\ll 1$.  On $B_\rho$,  we define
\begin{equation}\label{eq:Phi_delta}
\Phi_\delta(z)
:=\frac{|z|^2}{2}
+\chi\left(\frac{|z|}{r_\delta}\right)\varphi(z)
+\left(1-\chi\left(\frac{|z|}{r_\delta}\right)\right)\delta^2\psi\left(\frac{z}{\delta}\right).
\end{equation}
Note that $\varpi^*i\partial\bar{\partial} \Phi_\delta$ is a real $(1,1)$-form on $\widehat{M}$, which coincides with $\omega^\delta_{BS}$ (resp. $\varpi^*\omega$) on $\varpi^{-1}(B_{r_\delta})$ (resp. $\varpi^{-1}(B_\rho\setminus B_{2r_\delta})$). The glued metric $g_\delta$ is defined through
\begin{equation}\label{eq:glue_Kahler}
\omega_\delta:=\begin{cases}
\varpi^*i\partial\bar{\partial} \Phi_\delta,\ \ \ &\text{on }\varpi^{-1}(B_\rho),\\
\varpi^*\omega,\ \ \ &\text{on }\widehat{M}\setminus\varpi^{-1}(B_\rho),
\end{cases}
\end{equation}
provided that  $\omega_\delta$ is positive on the gluing annulus
\[
A_\delta:=\varpi^{-1}(B_{2r_\delta}\setminus B_{r_\delta}).
\]
To see this,  let us use rescaled coordinates $w=z/r_\delta$.  Consider the holomorphic map
\[
\iota:\{0<|w|\leq\rho/r_\delta\}\xrightarrow{z=r_\delta w}\left\{0<|z|\leq \rho\right\}\xrightarrow{\varpi^{-1}}\widehat{M}
\]
and define the $(1,1)$-tensor
\[
\widetilde{g}_\delta:=r_\delta^{-2}\iota^*g_\delta
\]
on $\{0<|w|\leq\rho/r_\delta\}$. With respect to the coordinate $w$, the K\"{a}hler form $\widetilde{\omega}_\delta$ of $\widetilde{g}_\delta$ can be written as
\[
\widetilde{\omega}_\delta=i\partial\bar{\partial}\,\widetilde{\Phi}_\delta,
\]
where
\[
\widetilde{\Phi}_\delta(w)=\frac{|w|^2}{2}
+\chi\left(|w|\right)r_\delta^{-2}\varphi(r_\delta w)
+\bigl(1-\chi\left(|w|\right)\bigr)r_\delta^{-2}\delta^2\psi\left(\frac{r_\delta w}{\delta}\right),\ \ \ 0<|w|\leq\rho/r_\delta.
\]
Denote by $g_{\mathrm{eucl}}$  the Euclidean metric on $\mathbb{C}^n$ (whose K\"ahler form is given by $i\partial\bar{\partial}(|w|^2/2)$).

\begin{lemma}\label{lm:shell}
Set
\[
A:=\left\{\frac{1}{2}\le |w|\le 3\right\}\subset \{0<|w|\leq\rho/r_\delta\}.
\]
Then
\[
\|\widetilde{g}_\delta-g_{\mathrm{eucl}}\|_{C^2(A)}=O(r_\delta^2),\qquad \delta\rightarrow0.
\]
In particular, $\widetilde{g}_\delta$ converges to $g_{\mathrm{eucl}}$ in the $C^2$ topology on $A$.
\end{lemma}

\begin{proof}
Set
\[
\Theta_\delta(w)
:=\widetilde{\Phi}_\delta(w)-\frac{|w|^2}{2}
:=\chi(|w|)\,r_\delta^{-2}\varphi(r_\delta w)
+\bigl(1-\chi(|w|)\bigr)r_\delta^{-2}\delta^2\psi\left(\frac{r_\delta w}{\delta}\right).
\]
It suffices to show that
\begin{equation}\label{eq:C4_converge}
\|\Theta_\delta\|_{C^4(A)}:=\sup_{w\in A}\sum^4_{k=1}|D^k_w\Theta_\delta(w)|=O(\delta^{2\alpha})=O(r_\delta^2)
\end{equation}
as $\delta\rightarrow0$. We have
\begin{equation}\label{eq:C4_converge_0}
|D^k_w\Theta_\delta(w)|\lesssim \sum^k_{j=0}r_\delta^{-2+j}|D^j_z\varphi(r_\delta w)|+\sum^k_{j=0}\left(\frac{r_\delta}{\delta}\right)^{-2+j}\left|D^j_\zeta\psi\left(\frac{r_\delta w}{\delta}\right)\right|,\ \ \ 0\leq k\leq 4,
\end{equation}
where $\zeta=z/\delta=r_\delta w/\delta$ and the implicit constant is numerical. Since $\varphi(z)=O(|z|^4)$ as $z\rightarrow0$, we have $|D_z^j\varphi(r_\delta w)|=O(r_\delta^{4-j})$ as $\delta\rightarrow0$ uniformly on $A$. Thus
\begin{equation}\label{eq:C4_converge_1}
\sum^k_{j=0}r_\delta^{-2+j}|D^j_z\varphi(r_\delta w)|=O(r_\delta^2)=O(\delta^{2\alpha}).
\end{equation}
On the other hand, for $n\ge 3$, \eqref{eq:BS_infty} gives
\[
\left|D^j_\zeta\psi\left(\frac{r_\delta w}{\delta}\right)\right|
=   O\left(\left|\frac{r_\delta w}{\delta}\right|^{4-2n-j}\right)=O\left(\left(\frac{r_\delta}{\delta}\right)^{4-2n-j}\right)
\]
uniformly on $A$,  so that
\begin{equation}\label{eq:C4_converge_2}
\sum^k_{j=0}\left(\frac{r_\delta}{\delta}\right)^{-2+j}\left|D^j_\zeta\psi\left(\frac{r_\delta w}{\delta}\right)\right|
:=O\left(\left(\frac{r_\delta}{\delta}\right)^{2-2n}\right)=O(\delta^{(2n-2)(1-\alpha)}).
\end{equation}
For $n=2$, we have $\psi(\zeta)=\log|\zeta|^2$, and for $j\geq1$,
\[
\left|D^j_\zeta\psi\left(\frac{r_\delta w}{\delta}\right)\right|=O\left(\left|\frac{r_\delta w}{\delta}\right|^{-j}\right)=O\left(\left(\frac{r_\delta}{\delta}\right)^{-j}\right)
\]
holds uniformly on $A$. Thus
\begin{eqnarray}\label{eq:C4_converge_3}
\sum^k_{j=0}\left(\frac{r_\delta}{\delta}\right)^{-2+j}\left|D^j_\zeta\psi\left(\frac{r_\delta w}{\delta}\right)\right|
&=& \left(\frac{r_\delta}{\delta}\right)^{-2}\left|\log\left|\frac{r_\delta w}{\delta}\right|^2\right|+\sum^k_{j=1}\left(\frac{r_\delta}{\delta}\right)^{-2+j}\left|D^j_\zeta\psi\left(\frac{r_\delta w}{\delta}\right)\right|\notag\\
&=& O\left(\delta^{2(1-\alpha)}\log\frac{1}{\delta}\right)+O\left(\delta^{2(1-\alpha)}\right)\notag\\
&=& O\left(\delta^{2(1-\alpha)}\log\frac{1}{\delta}\right).
\end{eqnarray}
By \eqref{eq:C4_converge_0}, \eqref{eq:C4_converge_1}, \eqref{eq:C4_converge_2} and
\eqref{eq:C4_converge_3}, we have
\[
\|\Theta_\delta\|_{C^4(A)}=
\begin{cases}
O(\delta^{2\alpha})+O(\delta^{(2n-2)(1-\alpha)}),\ \ \ &n\geq3,\\
O(\delta^{2\alpha})+O(\delta^{2(1-\alpha)}\log(1/\delta)),\ \ \ &n=2.
\end{cases}
\]
Since $0<\alpha<1/2$,  the assertion immediately follows. 
\end{proof}

As a consequence, $\omega_\delta=r_\delta^2(\iota^{-1})^*\widetilde{\omega}_\delta$ is positive on $A_\delta$ for all sufficiently small $\delta$. Hence $g_\delta$ is a genuine K\"{a}hler metric on $\widehat M$.

\medskip
\noindent\textbf{Step 2: Gluing estimates.}

Let us first show the following
\begin{lemma}\label{lm:BS_bounded}
$\omega_{BS}^\delta$ is uniformly bounded on $\varpi^{-1}(B_\rho)$ as $\delta\rightarrow0$.
\end{lemma}
\begin{proof}
Since
\[
\omega_{BS}^\delta=\delta^2a\cdot \pi^*\omega_{FS}+\varpi^*\left(i\partial\bar{\partial}\left(\frac{|z|^2}{2}+\delta^2\eta(|z/\delta|^2)\right)\right)
\]
in view of \eqref{eq:BS_FS}, it suffices to verify that $\varpi^\ast i\partial\bar{\partial}(\delta^2\eta(|z/\delta|^2))$ is uniformly bounded on $\varpi^{-1}(B_\rho)$ as $\delta\rightarrow0$.  If $n=2$,  then $\eta\equiv0$ and there is nothing to proof.  Suppose  $n\geq 3$ and fix $C\gg1$. Since $\eta$ is a smooth function on $[0,+\infty)$, it follows that
\[
i\partial\bar{\partial}(\delta^2\eta(|z/\delta|^2))
=\left.\left(i\partial_\zeta\bar{\partial}_\zeta\eta(|\zeta|^2)\right)\right|_{\zeta=z/\delta}
\]
is uniformly bounded for $|z|\leq C\delta$ as $\delta\rightarrow0$.  Moreover,  since $\eta(|\zeta|^2)=\psi(\zeta)-a\log|\zeta|^2$,  we may choose $C\gg 1$ such that
\[
i\partial_\zeta\bar{\partial}_\zeta\psi(\zeta)=O(|\zeta|^{2-2n})
\]
when $|\zeta|=|z/\delta|\geq C$, in view of \eqref{eq:BS_infty}.  Thus
\[
i\partial\bar{\partial}(\delta^2\eta(|z/\delta|^2))=\left.\left(i\partial_\zeta\bar{\partial}_\zeta\psi(\zeta)\right)\right|_{\zeta=z/\delta}-a\delta^2\cdot i\partial\bar{\partial}\log|z|^2=\delta^{2n-2}O(|z|^{2-2n})+\delta^2O(|z|^{-2}),
\]
which is uniformly bounded for $C\delta\leq |z|\leq \rho$ as $\delta\rightarrow0$.  This completes the proof.
\end{proof}
Now we give the crucial estimates.

\begin{lemma}\label{lm:gluing_estimates}
There exists a constant $C_0>0$, independent of $\delta$, such that the following hold.
\begin{eqnarray}
& & S_C(g_\delta)=0\ \ \ \text{on}\ \varpi^{-1}(B_{r_\delta}),\label{eq:sc_glue_1}\\
& & |S_C(g_\delta)|\leq C_0\ \ \ \text{on}\ A_\delta=\varpi^{-1}(B_{2r_\delta}\setminus B_{r_\delta}),\label{eq:sc_glue_2}\\
& & S_C(g_\delta)=S_C(g)\ \ \ \text{on}\ \widehat{M}\setminus \varpi^{-1}(B_{2r_\delta}),\label{eq:sc_glue_3}\\
& & \mathrm{vol}_{g_\delta}(\widehat{M})\leq \mathrm{vol}_g(M)+O(r_\delta^2),\qquad \delta\to0,\label{eq:volume_glue}\\
& & \int_{\widehat{M}}|S_C^-(g_\delta)|^n dV_{g_\delta}\leq \int_M|S_C^-(g)|^n dV_g+O(r_\delta^{2n}),\qquad \delta\to0,\label{eq:I_C_minus_glue}\\
& & \int_{\widehat{M}}|S_C(g_\delta)|^n dV_{g_\delta}\leq \int_M|S_C(g)|^n dV_g+O(r_\delta^{2n}),\qquad \delta\to0.\label{eq:I_C_glue}
\end{eqnarray}
\end{lemma}

\begin{proof}
\eqref{eq:sc_glue_1} and \eqref{eq:sc_glue_3} are direct consequences of  the definition of $g_\delta$.  For \eqref{eq:sc_glue_2},  we infer from  Lemma \ref{lm:shell} that
\[
\sup_A|S_C(\widetilde{g}_\delta)|=\sup_A|S_C(\widetilde{g}_\delta)-S_C(g_{\mathrm{eucl}})|=O(r_\delta^2),\qquad \delta\to0, 
\]
which implies
\[
\sup_{A_\delta}|S_C(g_\delta)|=\frac{1}{r_\delta^2}\sup_{1\leq|w|\leq2}|S_C(\widetilde{g}_\delta)|=\frac{1}{r_\delta^2}O(r_\delta^2)=O(1),\qquad \delta\to0.
\]

Since $g_\delta$  differs from $g$ only inside $\varpi^{-1}(B_{2r_\delta})$,  it follows that
\begin{equation}\label{eq:volume_glue_1}
\mathrm{vol}_{g_\delta}(\widehat{M}\setminus \varpi^{-1}(B_{2r_\delta}))=\mathrm{vol}_g(M\setminus \varpi^{-1}(B_{2r_\delta}))\leq \mathrm{vol}_g(M).
\end{equation}
Recall that $\widetilde{g}_\delta=r_\delta^{-2}\iota^*g_\delta$ and $\widetilde{g}_\delta\rightarrow g_{\mathrm{eucl}}$ in the $C^2$ topology on $A$.  Thus
\[
\mathrm{vol}_{g_\delta}(\varpi^{-1}(B_{2r_\delta}\setminus B_{r_\delta}))=r_\delta^{2n}\mathrm{vol}_{\widetilde{g}_\delta}(\{1\leq|w|<2\})
\]
and
\[
\lim_{\delta\rightarrow0}\mathrm{vol}_{\widetilde{g}_\delta}(\{1\leq|w|<2\})=\mathrm{vol}_{g_{\mathrm{eucl}}}(\{1\leq|w|<2\}),
\]
so that
\begin{equation}\label{eq:volume_glue_2}
\mathrm{vol}_{g_\delta}(\varpi^{-1}(B_{2r_\delta}\setminus B_{r_\delta}))=O(r_\delta^{2n}),\qquad \delta\to0.
\end{equation}
Moreover, as $\varpi^{-1}(B_{r_\delta})$ is a tubular neighbourhood of the exceptional locus in $\widehat{M}$, we have
\begin{equation}\label{eq:volume_glue_3}
\mathrm{vol}_{g_\delta}(\varpi^{-1}(B_{r_\delta}))=O(r_\delta^2),\qquad \delta\to0,
\end{equation}
in view of Lemma \ref{lm:BS_bounded}. 
Thus  \eqref{eq:volume_glue} follows immediately form \eqref{eq:volume_glue_1}, \eqref{eq:volume_glue_2} and  \eqref{eq:volume_glue_3}.  
Similarly,    \eqref{eq:sc_glue_1}--\eqref{eq:sc_glue_3} yield \eqref{eq:I_C_minus_glue} and
\eqref{eq:I_C_glue}.
\end{proof}

\medskip
\noindent\textbf{Step 3: Equalities for $\mathcal{I}_C$ and $\mathcal{I}_C^-$.}

Suppose that $K_M$ is nef.  Recall that $g$ satisfies \eqref{eq:approximating_metric_I_minus}.  By \eqref{eq:I_C_minus_glue}, we get
\[
\int_{\widehat{M}}|S_C^-(g_\delta)|^n dV_{g_\delta}
\leq \int_M|S_C^-(g)|^n dV_g+O(r_\delta^2)
\leq \frac{(n\pi)^n}{n!}\mathrm{CanVol}(M)+\varepsilon+O(r_\delta^2),
\]
so that
\[
\mathcal{I}^-_C(\widehat{M})\leq \frac{(n\pi)^n}{n!}\mathrm{CanVol}(M)+\varepsilon+O(r_\delta^2).
\]
Letting $\delta\rightarrow0$ and $\varepsilon\rightarrow0$,  we obtain
\[
\mathcal{I}^-_C(\widehat{M})\leq \frac{(n\pi)^n}{n!}\mathrm{CanVol}(M).
\]
This combined with Theorem \ref{th:Main0} gives
\[
\mathcal{I}^-_C(\widehat{M})= \frac{(n\pi)^n}{n!}\mathrm{CanVol}(M).
\]
Using \eqref{eq:approximating_metric_I} and \eqref{eq:I_C_glue} instead of \eqref{eq:approximating_metric_I_minus} and \eqref{eq:I_C_minus_glue},  we obtain the equality for  $\mathcal{I}_C(M)$.

\medskip
\noindent\textbf{Step 4: Equality for $\mathrm{MinVol}_C$.}

This part is the most involved.  Suppose  $K_M$ is nef.  It is known from  \eqref{eq:approximating_metric_vol} that for any $\varepsilon>0$ there exists a K\"{a}hler metric $g$ with
$S_C(g)\ge -1$ and
\[
\mathrm{vol}_g(M)\le \frac{(n\pi)^n}{n!}\,\mathrm{CanVol}(M)+\varepsilon.
\]
Fix $\tau\in(0,1)$. Replacing $g$ by $\lambda g$ with $\lambda=(1-\tau)^{-1}>1$, we get
\[
S_C(g)\ge -1+\tau
\]
and
\[
\mathrm{vol}_g(M)\le (1-\tau)^{-n}\left(\frac{(n\pi)^n}{n!}\,\mathrm{CanVol}(M)+\varepsilon\right).
\]
If $\tau$ is sufficiently small,  then
\[
\mathrm{vol}_g(M)\le \frac{(n\pi)^n}{n!}\,\mathrm{CanVol}(M)+2\varepsilon.
\]
This combined with  \eqref{eq:volume_glue} implies
\begin{equation}\label{eq:volume_canvol_epsilon}
\mathrm{vol}_{g_\delta}(\widehat{M})\leq \frac{(n\pi)^n}{n!}\,\mathrm{CanVol}(M)+3\varepsilon
\end{equation}
when $\delta$ is sufficiently small. However, this is not enough to verify the equality for $\mathrm{MinVol}_C$, for we merely have
\begin{equation}\label{eq:sc_glue}
S_C(g_\delta)\geq
\begin{cases}
0\ \ \ \ \ &\text{on}\ \varpi^{-1}(B_{r_\delta}),\\
-C_0\ \ \ \ \ &\text{on}\ A_\delta=\varpi^{-1}(B_{2r_\delta}\setminus B_{r_\delta}),\\
-1+\tau\ \ \ \ \ &\text{on}\ \widehat{M}\setminus \varpi^{-1}(B_{2r_\delta}),
\end{cases}
\end{equation}
in view of \eqref{eq:sc_glue_1}, \eqref{eq:sc_glue_2} and \eqref{eq:sc_glue_3}. It is unclear whether $C_0\leq1$.

To overcome this difficulty, we will use again the technique of conformal deformation. Set $h_\delta=e^{u_\delta/n}g_\delta$ for a real-valued smooth function $u_\delta$ on $\widehat M$. Recall that
\begin{equation}\label{eq:sc_conformal}
S_C(h_\delta)=e^{-u_\delta/n}\left(S_C(g_\delta)-\Delta_{g_\delta} u_\delta\right),
\end{equation}
where $\Delta_{g_\delta}$ denotes the complex Laplacian with respect to the metric $g_\delta$. We are going to prove the following

\begin{proposition}\label{prop:conformal_raise}
For $0<\delta\ll\rho\ll1$, there exists
$u_\delta\in C^\infty(\widehat M)$ with $u_\delta\ge 0$, such that the Hermitian metric $h_\delta=e^{u_\delta/n}g_\delta$ satisfies
\begin{eqnarray*}
& & S_C(h_\delta)\ge -1,\\
& & \mathrm{vol}_{h_\delta}(\widehat M)=\mathrm{vol}_{g_\delta}(\widehat M)+O(r_\delta^2),\qquad \delta\rightarrow0.
\end{eqnarray*}
\end{proposition}

We first show the following

\begin{lemma}\label{lm:Laplacian_scaling}
Let $\varphi$ be a smooth function on $\{0<|w|\leq \rho/r_\delta\}$. Set $\varphi_{r_\delta}(z)=\varphi(z/r_\delta)$ for $0<|z|\leq\rho$, which can be identified with a smooth function on $\varpi^{-1}(B_\rho\setminus\{0\})\subset\widehat{M}$. Then
\[
\Delta_{g_\delta}\varphi_{r_\delta}(z)=r_\delta^{-2}\Delta_{\widetilde{g}_\delta}\varphi(w)|_{w=z/r_\delta}.
\]
\end{lemma}

\begin{proof}
If one writes
\[
g_\delta=\sum(g_\delta)_{j\bar{k}}(z)dz_j\otimes d\bar{z}_k,
\]
then
\[
\widetilde{g}_\delta=r_\delta^{-2}\iota^*g_\delta=\sum(g_\delta)_{j\bar{k}}(r_\delta w)dw_j\otimes d\bar{w}_k.
\]
Let  $((g_\delta)^{j\bar{k}})$ denote the inverse matrix of $((g_\delta)_{j\bar{k}})$.   Then we have
\begin{eqnarray*}
\Delta_{g_\delta}\varphi_{r_\delta}(z)
&=& \sum (g_\delta)^{j\bar{k}}(z)\frac{\partial^2\varphi_{r_\delta}}{\partial z_j\bar{z}_k}(z)\\
&=& r_\delta^{-2}\sum (g_\delta)^{j\bar{k}}(z)\left.\frac{\partial^2\varphi}{\partial w_j\bar{w}_k}(w)\right|_{w=z/r_\delta}\\
&=& r_\delta^{-2}\left.\left(\sum (g_\delta)^{j\bar{k}}(r_\delta w)\frac{\partial^2\varphi}{\partial w_j\bar{w}_k}(w)\right)\right|_{w=z/r_\delta}\\
&=& r_\delta^{-2}\Delta_{\widetilde{g}_\delta}\varphi(w)|_{w=z/r_\delta}.
\end{eqnarray*}
\end{proof}

Pick a smooth radial function $f\in C_0^\infty(\mathbb C^n)$ such that
\begin{eqnarray*}
& & -C_0\le f\le 0 \quad \text{on } \left\{\frac{1}{2}\le |w|\le 3\right\},\\
& & f\equiv -C_0 \quad \text{on } \{1\le |w|\le 2\},\\
& & \mathrm{supp}\,f\subset \left\{\frac{1}{2}\le |w|\le 3\right\}.
\end{eqnarray*}
Let $N(w)=-c_n|w|^{2-2n}$ be the Newtonian kernel in real dimension $2n\ge 4$,  and define
\[
v:=N*f.
\]
Then $v$ is smooth and radial on $\mathbb C^n$, satisfying $v\geq0$ and
\[
\Delta_{\mathrm{eucl}}v:=2\sum^n_{j=1}\frac{\partial^2v}{\partial w_j\partial \bar{w}_j}=f.
\]
We also need the following elementary fact.

\begin{lemma}\label{lm:v_bdd}
There exists a constant $C$ depending on $n$ and $C_0$ such that
\begin{eqnarray*}
& & |v|\leq C,\qquad \left|\frac{\partial v}{\partial w_j}\right|\leq C,\ \ \ 1\leq j\leq n,\\
& & \left|\frac{\partial^2 v}{\partial w_j\partial\bar{w}_k}\right|\leq C,\ \ \ 1\leq j,k\leq n.
\end{eqnarray*}
\end{lemma}

\begin{proof}
We have
\[
|v(w)|=|N*f|\leq \|N\|_{L^1(w+\mathrm{supp}\,f)}\|f\|_{L^\infty(\mathbb{C}^n)}\le C,\ \ \forall\,w\in \mathbb C^n.
\]
Similarly,
\begin{eqnarray*}
& & \left|\frac{\partial v}{\partial w_j}(w)\right|
=\left|N*\frac{\partial f}{\partial w_j}\right|
\leq \|N\|_{L^1(w+\mathrm{supp}\,f)}\left\|\frac{\partial f}{\partial w_j}\right\|_{L^\infty(\mathbb{C}^n)}\le C,\\
& & \left|\frac{\partial^2 v}{\partial w_j\partial\bar{w}_k}(w)\right|
=\left|N*\frac{\partial^2 f}{\partial w_j\partial\bar{w}_k}\right|
\leq \|N\|_{L^1(w+\mathrm{supp}\,f)}\left\|\frac{\partial^2 f}{\partial w_j\partial\bar{w}_k}\right\|_{L^\infty(\mathbb{C}^n)}\le C.
\end{eqnarray*}
\end{proof}

\begin{proof}[Proof of Proposition \ref{prop:conformal_raise}]
The function $v$ constructed above is harmonic on $\{|w|<1/2\}\cup\{|w|>3\}$. Since $v$ is radial, it  has to be a  constant on $\{|w|<1/2\}$.  Take a cut-off function $\kappa\in C_0^\infty(B_\rho)$ such that $\kappa\equiv 1$ on $B_{3r_\delta}$ and 
\begin{equation}\label{eq:kappa_bdd}
\left|\frac{\partial \kappa}{\partial z_j}\right|\lesssim 1/\rho,\qquad \left|\frac{\partial^2 \kappa}{\partial z_j\partial\bar{z}_k}\right|\lesssim 1/\rho^2,\qquad 1\leq j,k\leq n,
\end{equation}
where the implicit constants are independent of $\rho$ and $\delta$ (note that $3r_\delta\ll\rho$).  Now define $u_\delta\in C^\infty_0(\widehat{M})$ as
\[
u_\delta:=
\begin{cases}
r_\delta^2\,\kappa(z)\,v(z/r_\delta),\ \ \ &\text{on }\varpi^{-1}(B_\rho\setminus\{0\}),\\
0,\ \ \ &\text{on }\widehat{M}\setminus\varpi^{-1}(B_\rho\setminus\{0\}).
\end{cases}
\]
Recall that $z$ is a local coordinate on $B_\rho$, hence we may identify a point in $\varpi^{-1}(B_\rho\setminus\{0\})$ with its coordinate in $B_\rho\setminus \{0\}$. The function $u_\delta$ can be extended smoothly across the exceptional divisor since $v(z/r_\delta)$ and $\kappa$ are constants on $\varpi^{-1}(B_{r_\delta/2}\setminus\{0\})$.

In what follows, we will estimate the Laplacian of $u_\delta$ with respect to $g_\delta$.

\begin{itemize}
\item[$(1)$]
On $\widehat{M}\setminus\varpi^{-1}(B_\rho)$, we have $u_\delta\equiv0$, so that $\Delta_{g_\delta}u_\delta=0$.

\item[$(2)$]
On $\varpi^{-1}(B_{r_\delta/2})$,  both $v(z/r_\delta)$ and $\kappa(z)$ are  constants,  so that $\Delta_{g_\delta}u_\delta=0$.

\item[$(3)$]
On $\varpi^{-1}(B_{3r_\delta}\setminus B_{r_\delta/2})$, we have $u_\delta(z)=r_\delta^2v(z/r_\delta)$. It follows from Lemma \ref{lm:Laplacian_scaling} that
\[
\Delta_{g_\delta}u_\delta(z)=\Delta_{\widetilde{g}_\delta}v(w)|_{w=z/r_\delta}.
\]
Lemma \ref{lm:shell} implies 
\[
\sup_{A}\left|\Delta_{\widetilde{g}_\delta}v-\Delta_{\mathrm{eucl}}v\right|\rightarrow0
\]
where $A=\{1/2\leq|w|\leq3\}$,  so that
\[
\Delta_{g_\delta}u_\delta(z)\rightarrow \Delta_{\mathrm{eucl}}v(w)|_{w=z/r_\delta}=f(z/r_\delta)
\]
uniformly for $r_\delta/2\leq |z|\leq 3r_\delta$ as $\delta\rightarrow0$. Thus we may take $\delta\ll 1$ such that
\[
\sup_{r_\delta/2\leq |z|\leq 3r_\delta}\left|\Delta_{g_\delta}u_\delta(z)-f(z/r_\delta)\right|\leq\tau\leq 1.
\]
By the definition of $f$, we obtain
\begin{eqnarray*}
& & \Delta_{g_\delta}u_\delta\leq 1\ \ \ \text{on}\ \varpi^{-1}(B_{r_\delta}\setminus B_{r_\delta/2}),\\
& & \Delta_{g_\delta}u_\delta\leq -C_0+1\ \ \ \text{on}\ \varpi^{-1}(B_{2r_\delta}\setminus B_{r_\delta}),\\
& & \Delta_{g_\delta}u_\delta\leq \tau\ \ \ \text{on}\ \varpi^{-1}(B_{3r_\delta}\setminus B_{2r_\delta}).
\end{eqnarray*}

\item[$(4)$]
On $\varpi^{-1}(B_\rho\setminus B_{3r_\delta})$, since $g_\delta=g$, we have
\begin{eqnarray}\label{eq:Laplacian_case_4}
\Delta_{g_\delta}u_\delta(z)
&=& r_\delta^2v(z/r_\delta)\Delta_{g_\delta}\kappa(z)
+\kappa(z)\Delta_{g_\delta}\left(r_\delta^2v(z/r_\delta)\right)
+2r_\delta^2\mathrm{Re\,}\mathrm{tr}_{g_\delta}\left(i\partial\kappa\wedge\bar{\partial}_zv(z/r_\delta)\right)\notag\\
&=& r_\delta^2v(z/r_\delta)\Delta_g\kappa(z)
+\kappa(z)\Delta_{\widetilde{g}_\delta}v(w)|_{w=z/r_\delta}
+2r_\delta\mathrm{Re\,}\mathrm{tr}_{g}\left(i\partial\kappa\wedge\bar{\partial}_wv(z/r_\delta)\right).
\end{eqnarray}
By \eqref{eq:kappa_bdd},
\begin{equation}\label{eq:Laplacian_case_4_1}
\left|r_\delta^2v(z/r_\delta)\Delta_g\kappa(z)\right|\lesssim\frac{r_\delta^2}{\rho^2},\qquad \forall\,z\in B_\rho\setminus B_{3r_\delta}.
\end{equation}
Here and in what follows, the implicit constants are independent of $\rho$ and $\delta$.  Note that the coefficients of $\widetilde{g}_\delta$ are given by
\[
(\widetilde{g}_\delta)_{j\bar{k}}(w)=g_{j\bar{k}}(r_\delta w)=\frac{1}{2}\delta_{j\bar{k}}+\frac{\partial^2\varphi}{\partial z_j\partial\bar{z}_k}(r_\delta w),\ \ \ 3\leq |w|\leq \rho/r_\delta.
\]
Since $\varphi(z)=O(|z|^4)$ on $\varpi^{-1}(B_\rho)$, we have
\[
\sup_{3\leq |w|\leq \rho/r_\delta}\left|(\widetilde{g}_\delta)_{j\bar{k}}(w)-\frac{1}{2}\delta_{j\bar{k}}\right|=O(\rho^2),\qquad \rho\rightarrow0.
\]
This together with Lemma \ref{lm:v_bdd} yield
\[
\sup_{3\leq |w|\leq \rho/r_\delta}\left|\Delta_{\widetilde{g}_\delta}v-\Delta_{\mathrm{eucl}}v\right|
=O(\rho^2),\ \ \ \rho\rightarrow0.
\]
Since $\Delta_{\mathrm{eucl}}v(w)=f(w)=0$ when $|w|\geq3$, it follows that
\begin{equation}\label{eq:Laplacian_case_4_2}
\left|\kappa(z)\Delta_{\widetilde{g}_\delta}v(w)|_{w=z/r_\delta}\right|\lesssim\rho^2,\qquad \forall\,z\in B_\rho\setminus B_{3r_\delta}.
\end{equation}
Moreover, since $g^{j\bar{k}}=\delta_{j\bar{k}}+O(\rho^2)$ as $\rho\rightarrow0$, we have
\begin{eqnarray}\label{eq:Laplacian_case_4_3}
\left|2r_\delta\mathrm{Re\,}\mathrm{tr}_{g}\left(i\partial\kappa\wedge\bar{\partial}_wv(z/r_\delta)\right)\right|
&=& 2r_\delta\left|\mathrm{Re}\,\sum g^{j\bar{k}}\frac{\partial\kappa}{\partial z_j}\frac{\partial v}{\partial\bar{w}_k}(z/r_\delta)\right|\notag\\
&\lesssim& \frac{r_\delta}{\rho},\qquad \forall\,z\in B_\rho\setminus B_{3r_\delta}.
\end{eqnarray}
when $0<\delta\ll\rho\ll1$, in view of Lemma \ref{lm:v_bdd} and \eqref{eq:kappa_bdd}. It follows from \eqref{eq:Laplacian_case_4_1}, \eqref{eq:Laplacian_case_4_2}, \eqref{eq:Laplacian_case_4_3} and \eqref{eq:Laplacian_case_4} that
\[
\Delta_{g_\delta}u_\delta(z)\lesssim \rho^2+\frac{r_\delta^2}{\rho^2}+\frac{r_\delta}{\rho},\qquad \forall\,z\in B_\rho\setminus B_{3r_\delta}.
\]
As a consequence, we may choose $0<\delta\ll\rho\ll1$ such that
\[
\Delta_{g_\delta}u_\delta\leq\tau\ \ \ \text{on }\ \varpi^{-1}(B_\rho\setminus B_{3r_\delta}).
\]
\end{itemize}

\vspace{5mm}

Combining (1)$\sim$(4) with \eqref{eq:sc_glue}, we conclude that
\begin{equation}\label{eq:sc_raise}
S_C(g_\delta)-\Delta_{g_\delta}u_\delta\geq-1.
\end{equation}
Since $u_\delta\ge 0$, it follows from \eqref{eq:sc_raise} and \eqref{eq:sc_conformal} that
\[
S_C(h_\delta)=e^{-u_\delta/n}\left(S_C(g_\delta)-\Delta_{g_\delta} u_\delta\right)\ge -1.
\]

It remains to compare $\mathrm{vol}_{h_\delta}$ with $\mathrm{vol}_{g_\delta}$.
Since $u_\delta=0$ on $\widehat{M}\setminus\varpi^{-1}(B_\rho)$,
\begin{equation}\label{eq:volume_conformal_1}
\mathrm{vol}_{h_\delta}(\widehat{M}\setminus\varpi^{-1}(B_\rho))=\mathrm{vol}_{g_\delta}(\widehat{M}\setminus\varpi^{-1}(B_\rho)).
\end{equation}
By Lemma \ref{lm:v_bdd}, we have $|u_\delta|\leq Cr_\delta^2$, so that
\[
\mathrm{vol}_{h_\delta}(\varpi^{-1}(B_\rho))
=\int_{\varpi^{-1}(B_\rho)} e^{u_\delta}\,dV_{g_\delta}
\leq e^{Cr_\delta^2}\mathrm{vol}_{g_\delta}(\varpi^{-1}(B_\rho)),
\]
while \eqref{eq:volume_glue} implies
\[
\mathrm{vol}_{g_\delta}(\varpi^{-1}(B_\rho))\leq \mathrm{vol}_{g_\delta}(\widehat{M})\leq \mathrm{vol}_g(\widehat{M})+O(r_\delta^2),\qquad \delta\rightarrow0;
\]
in particular, $\mathrm{vol}_{g_\delta}(\varpi^{-1}(B_\rho))$ is bounded by a constant independent of $\rho$ and $\delta$ as $\delta\rightarrow0$. Thus
\begin{equation}\label{eq:volume_conformal_2}
\mathrm{vol}_{h_\delta}(\varpi^{-1}(B_\rho))=\mathrm{vol}_{g_\delta}(\varpi^{-1}(B_\rho))+O(r_\delta^2).
\end{equation}
It follows from  \eqref{eq:volume_conformal_1} and \eqref{eq:volume_conformal_2} that
\[
\mathrm{vol}_{h_\delta}(\widehat{M})=\mathrm{vol}_{g_\delta}(\widehat{M})+O(r_\delta^2).
\]
\end{proof}

By Proposition \ref{prop:conformal_raise} and \eqref{eq:volume_canvol_epsilon},  we obtain
\[
\mathrm{MinVol}_C(M)\leq \mathrm{vol}_{h_\delta}(\widehat{M})\leq \frac{(n\pi)^n}{n!}\mathrm{CanVol}(M)+4\varepsilon
\]
when $0<\delta\ll\rho\ll1$, while  Theorem \ref{th:Main0} gives
\[
\mathrm{MinVol}_C(M)\geq \frac{(n\pi)^n}{n!}\mathrm{CanVol}(M).
\]
The proof of Theorem \ref{th:blow_up} is complete.

\section{Failure of  regularization of $g_{KE}$  with  Ricci curvatures bounded}\label{sec:Ricci}

It is also natural to consider
\[
\mathcal{I}_R(M):=\inf_{g\in\mathcal{RM}(M)}\int_M|Ric(g)|^{n/2}dV_g
\]
for any compact Riemannian $n-$manifold $M$ (see LeBrun \cite{LeBrun01}). By \cite[Theorem 4.3]{LeBrun01}, if $M$ is a projective surface with $K_M$ big and nef and $M_k$ is the blow-up of $M$ at $k$ points, then
\begin{equation}\label{eq:Ricci_integral}
\mathcal{I}_R(M_k)\geq 8\pi^2\mathrm{CanVol}(M)+k\ \left(>8\pi^2\mathrm{CanVol}(M_k)\right)
\end{equation}
holds.

Since $K_{M_k}$ is big, there exists a singular K\"{a}hler-Einstein metric $g_{KE}$ which is smooth on a Zariski open subset $U$ of $M_k$. Based on \eqref{eq:Ricci_integral}, we have

\begin{proposition}\label{prop:approximating_Ricci}
There does not exist a sequence $\{g_j\}$ of smooth K\"{a}hler metrics on $M_k$ such that the following conditions hold:
\begin{itemize}
\item[$(1)$]
$\{g_j\}$ converges in the $C^\infty_{\mathrm{loc}}-$topology to $g_{KE}$ on $U$.

\item[$(2)$]
Both $dV_{g_j}$ and $|Ric(g_j)|$ are uniformly bounded on $M_k$.
\end{itemize}
\end{proposition}

\begin{proof}
Let $\omega_{KE}$ be the K\"{a}hler form of $g_{KE}$ on $U$. Note that $Ric(\omega_{KE})=-\omega_{KE}$ on $U$. By \eqref{eq:lower_CanVol_2}, we have
\[
\mathrm{vol}_{KE}(M_k\setminus E)\leq 2\pi^2\mathrm{CanVol}(M_k),
\]
where $E:=M_k\setminus U$. Suppose on the contrary that there exists a sequence $\{g_j\}$ of K\"ahler metrics on $M_k$ satisfying (1) and (2).  Then for any $\delta>0$, there exists a neighbourhood $E_\delta$ of $E$ such that
\begin{itemize}
\item[(a)]
$\mathrm{vol}_{g_j}(E_\delta)<\delta$ for all $j$,

\item[(b)]
$|Ric(g_j)|\leq 2(1+\delta)$ on $M_k\setminus E_\delta$ when $j\gg1$, and

\item[(c)]
$\mathrm{vol}_{g_j}(M_k\setminus E_\delta)\leq 2\pi^2\mathrm{CanVol}(M_k)+\delta$ when $j\gg1$.
\end{itemize}
Here (b) follows from the fact that $|Ric(\omega_{KE})|^2=|\omega_{KE}|^2=2$,  $|Ric(g_{KE})|^2=2|Ric(\omega_{KE})|^2=4$ and $g_j\rightarrow g_{KE}$ uniformly on $M_k\setminus E_\delta$.

 Therefore,
\[
\int_{M_k\setminus E_\delta}|Ric(g_j)|^2dV_{g_j}\leq 4(1+\delta)^2\mathrm{vol}_{g_j}(M_k\setminus E_\delta)=4(1+\delta)^2(2\pi^2\mathrm{CanVol}(M_k)+\delta).
\]
Moreover, since $|Ric(g_j)|\leq C$ for some constant $C$ independent of $j$, we have
\[
\int_{E_\delta}|Ric(g_j)|^2dV_{g_j}\leq C^2\mathrm{vol}_{g_j}(E_\delta)<C^2\delta.
\]
Thus
\[
\mathcal{I}_R(M_k)\leq \int_{M_k} |Ric(g_j)|^2dV_{g_j}< 4(1+\delta)^2(2\pi^2\mathrm{CanVol}(M_k)+\delta)+C^2\delta,
\]
which contradicts \eqref{eq:Ricci_integral} when $\delta\ll0$.
\end{proof}

{\bf Acknowledgements.} The first author is supported by National Natural Science Foundation of China, No. 12271101; the third author is supported by National Natural Science Foundation of China, No. 12471079.


\begin{thebibliography}{99}

\bibitem{Aubin} T.  Aubin,  {\it \'{E}quations du type Monge-Amp$\grave{e}$re sur les vari\'et\'es k\"ah\'eriennes compactes},  Bull.  Sci.  Math.  (2) {\bf 102} (1978),  no. 1,  63--95.

\bibitem{BergerBook} M. Berger, A Panoramic View of Riemannian Geometry,  Springer-Verlag, Berlin,  2003.

\bibitem{Berman} R. Berman,  {\it Bergman kernels and local holomorphic Morse inequalities},  Math.  Z.  {\bf 248}  
 (2004),  325--344.

\bibitem{BerndtssonBook} B. Berndtsson, An introduction to things $\bar{\partial}$, Analytic and algebraic geometry, 7--76. IAS/Park City Math. Ser., 17, American Mathematical Society, Providence, RI, 2010.

\bibitem{BCG} G.  Besson,  G.  Courtois and S.  Gallot,   {\it Entropies et rigidit\'es des espaces
localement sym\'etriques de courbure strictement n\'egative},  Geom. Funct.  Anal.  {\bf 5} (1995),  
731--799.

\bibitem{BEGZ10} S. Boucksom, P. Eyssidieux, V. Guedj and A. Zeriahi, {\it Monge-Amp\`{e}re equations in big cohomology classes}, Acta Math., {\bf 205} (2010), 199--262.

\bibitem{Cao} H.-D.  Cao,  {\it Deformation of K\"ahler metrics to K\"ahler-Einstein metrics on compact K\"ahler manifolds}, Invent.  Math.  {\bf 81} (1985),  359--372.

\bibitem{Catlin} D.  Catlin, {\it The Bergman kernel and a theorem of Tian}, Analysis and Geometry in Several Complex Variables (Katata, 1997), Trends Math.,  Birkh\"ahser Boston,  Boston,  MA 1999, 1--23.

\bibitem{CheegerRong} J.  Cheeger and X.  Rong,  {\it Existence of polarized F-structures on collapsed manifolds with bounded curvature and diameter},  Geom.  Funct.  Anal.  {\bf 6} (1996),  
411-429.

\bibitem{Cutkosky24} S. D. Cutkosky, {\it The Minkowski equality of big divisors}, Michigan Math. J. {\bf 74} (2024), 451--483.

\bibitem{DemaillyBook} J.-P. Demailly, Analytic Methods in Algebraic Geometry,  Survey of Modern Mathematics, 1,
International Press, Somerville, MA; Higher Educational Press, Beijing, 2012.

\bibitem{EGZ}  P.  Eyssidieux,  V.  Guedj and A.  Zeriahi, {\it Singular K\"ahler-Einstein metrics},  J.  Amer.  Math.  Soc.  {\bf 22} (2009),  no. 3,  607--639.

\bibitem{GromovVolume} M.  Gromov,  {\it Volumes and bounded cohomology},  Publ.  Math.  IHES {\bf 56} (1982),  5--99.

\bibitem{Gromov4} M. Gromov,  Four Lectures on Scalar Curvature,  2019.

\bibitem{GPSS} B. Guo, D. H. Phong, J. Song, J. Sturm, {\it Diameter estimates in K\"{a}hler geometry}, Comm. Pure Appl. Math. {\bf 77} (2024), 3520--3556.

\bibitem{HM} C. Hacon and J. McKernan, {\it Boundedness of pluricanonical maps of varieties of general type}, Invent. Math. {\bf 166} (2006), 1--25.

\bibitem{Holschbach10} A. Holschbach, {\it A Chebotarev-type density theorem for divisors on algebraic varieties}, arXiv:1006.2340v2. 

\bibitem{KobayashiBook} S. Kobayashi, Differential Geometry of Complex Vector Bundles, Publ. Math. Soc. Japan, vol. 15, Princeton University Press, Princeton, NJ, 1987.

\bibitem{LazarsfeldBook} R. Lazarsfeld, Positivity in algebraic geometry. I, Ergeb. Math. Grenzgeb. (3), vol. 48, Springer-Verlag, Berlin, 2004.

\bibitem{LeBrun88} C. LeBrun, {\it Counter-examples to the generalized positive action conjecture}, Comm. Math. Phys. {\bf 188} (1988), 591--596.

\bibitem{LeBrun96} C. LeBrun, {\it Four-manifolds without Einstein metrics}, Math. Res. Lett. {\bf 3} (1996), 133--147.

\bibitem{LeBrun99} C. LeBrun, {\it Kodaira dimension and the Yamabe problem}, Comm. Anal. Geom. {\bf 7} (1999), 133--156.

\bibitem{LeBrun01}C. LeBrun, {\it Ricci Curvature, Minimal Volumes, and Seiberg-Witten Theory}, Invent. Math. {\bf 145} (2001), 279--316.

\bibitem{LeBrun23} C. LeBrun, {\it The scalar curvature of $4$-manifolds}, in Perspective in scalar curvature. Vol. 1, World Sci. Publ., Hackensack, NJ, 2023, 643--707.

\bibitem{LY17} K. Liu and X. Yang, {\it Ricci curvatures on Hermitian manifolds}, Trans. Amer. Math. Soc. {\bf 369} (2017), 5157--5196.

\bibitem{Liu} G. Liu, {\it Compact K\"{a}hler manifolds with nonpositive bisectional curvature}, Geom. Funct. Anal. {\bf 24} (2014), 1591--1607.

\bibitem{MM} X.  Ma and G.  Marinescu,  Holomorphic Morse Inequalities and Bergman Kernels,  Progress in Math.  
 {\bf 254},  Birkh\"auser,  2006.

\bibitem{Rong} X.  Rong,  {\it The existence of polarized $F-$structures on volume collapsed $4-$manifolds},  Geom. Funct.  Anal.  {\bf 3} (1993),  474--501.

\bibitem{Simanca91} S. R. Simanca, {\it K\"{a}hler metrics of constant scalar curvature on bundles over $\mathbb{CP}^{n-1}$}, Math. Ann. {\bf 291} (1991), 239--246.

\bibitem{Siu84} Y. T. Siu, {\it A vanishing theorem for semipositive line bundles over non-K\"{a}hler manifolds}, J. Differential Geom. {\bf 19} (1984), 431--452.

\bibitem{SW} J.  Song and B. Weinkove,  Lecture notes on the K\"ahler-Ricci flow,  arXiv: 1212.3653v1. 

\bibitem{ST12} J. Song and G. Tian, {\it Canonical measures and K\"{a}hler-Ricci flow}, J. Amer. Math. Soc. {\bf 25} (2012), 303--353.

\bibitem{ST16} J. Song and G. Tian, {\it Bounding scalar curvature for global solutions of the K\"{a}hler-Ricci flow}, Amer. J. Math., {\bf 138} (2016), 683--695.

\bibitem{Szekelyhidi12} G. Sz\'{e}kelyhidi, {\it On blowing up extremal Kahler manifolds},
Duke Math. J. {\bf 161} (2012), 1411--1453.

\bibitem{SzekelyhidiBook} G. Sz\'{e}kelyhidi, An Introduction to extremal K\"{a}hler metrics, American Mathematical Society, Providence, RI, 2014.

\bibitem{Takayama} S. Takayama, {\it Pluricanonical systems on algebraic varieties of general type}, Invent. Math. {\bf 165} (2006), 551--587.

\bibitem{Tian} G.  Tian,  {\it On a set of polarized K\"ahler metrics on algebraic manifolds},  J.  Diff.  Geom.  {\bf 32}  (1990),  99--130.

\bibitem{TZ06} G. Tian and Z. Zhang, {\it On the K\"{a}hler-Ricci flow on projective manifolds of general type}, Chinese Ann. Math. Ser. B {\bf 27} (2006), 179--192.

\bibitem{Tsuji88} H. Tsuji, {\it Existence and degeneration of K\"{a}hler-Einstein metrics on minimal algebraic varieties of general type}, Math. Ann. {\bf 281} (1988), 123--133.

\bibitem{Tsuji92} H. Tsuji, {\it Analytic Zariski Decomposition}, Proc. of Japan Acad. {\bf 61} (1992), 161--163.

\bibitem{Ueno} K. Ueno,  Classification theory of algebraic varieties and compact complex spaces,  Lecture Notes in Math., Vol. 439, Springer-Verlag, Berlin-New York, 1975. 

\bibitem{WZ} H.  Wu and F.  Zheng,  {\it Compact K\"ahler manifolds with nonpositive bisectional curvature},  J.  Diff. Geom.  {\bf 61} (2002),  263--287.

\bibitem{Yang21} X. Yang, {\it RC-positivity, vanishing theorems and rigidity of holomorphic maps}, J. Inst. Math. Jussieu {\bf 20} (2021), 1023--1038.

\bibitem{Yau} S.-T. Yau,  {\it On the Ricci curvature of a compact K\"ahler manifold and the complex Monge-Amp$\grave{e}$re equation, I},  Comm.  Pure Appl.  Math.  {\bf 31} (1978),  339--411.

\bibitem{WZhang} W.  Zhang,   {\it Geometric structures,  Gromov norm and Kodaira dimensions},   Adv.  Math.  {\bf 308} (2017),   
 1--35. 

\bibitem{Zhang09} Z. Zhang, {\it Scalar curvature bound for K\"{a}hler-Ricci flows over minimal manifolds of general type}, Int. Math. Res. Not. IMRN {\bf 20} (2009), 3901--3912.

\bibitem{Zelditch} S.  Zelditch,  {\it Szeg\"o kernels and a theorem of Tian},  Internat.  Math.  Res.  Notices {\bf 6} (1998),  317--331.

\end{thebibliography}
\end{document}